\documentclass[12pt]{amsart}

\newcommand{\rr}{r}

\newcommand{\R}{{\mathbb R}}

\newcommand{\Z}{{\mathbb Z}}

\newcommand{\half}{{\frac{1}{2}}}

\renewcommand{\phi}{\varphi}

\newcommand{\dcal}{\mathcal{D}}

\newcommand{\lcal}{\mathcal{L}}

\newtheorem{maintheo}{{\sc Theorem}}

\newtheorem{maincor}{{\sc Corollary}}

\newtheorem{theo}{{\sc Theorem}}
\newtheorem{cor}[theo]{{\sc Corollary}}
\newtheorem{lem}[theo]{{\sc Lemma}}
\newtheorem{prop}[theo]{{\sc Proposition}}

\newenvironment{rem}{\medskip\noindent{\it Remark:\/} }{\medskip}

\title[$C^{\infty} $ spectral rigidity of the ellipse]
{$C^{\infty} $ spectral Rigidity of the ellipse}

\author{Hamid Hezari }
\address{Department of Mathematics, MIT, Cambridge, MA 02139, USA} \email{hezari@math.mit.edu}
\author{Steve Zelditch}
\address{Department of Mathematics, Northwestern  University,
Evanston, IL 60208-2370, USA} \email{
zelditch@math.northwestern.edu}

\thanks{The first author is partially supported by the NSF
grant DMS-0969745 and the second author is partially supported by the NSF grant DMS-0904252}

\date{\today}

\begin{document}

\begin{abstract} We prove that ellipses are infinitesimally
spectrally rigid among $C^{\infty}$ domains with the symmetries of
the ellipse.
\end{abstract}

\maketitle

An isospectral deformation of a plane domain $\Omega_0$  is a
one-parameter family $\Omega_{\epsilon}$ of plane domains for
which the spectrum of the Euclidean Dirichlet (or Neumann)
Laplacian $\Delta_{\epsilon}$ is constant (including
multiplicities). We say that $\Omega_{\epsilon}$ is a $C^1$ curve
of $C^{\infty}$ plane domains if there exists a $C^1$ curve  of
diffeomorphisms $\phi_{\epsilon}$ of a neighborhood of $\Omega_0
\subset \R^2$  with $\phi_0 = id$ and with $ \Omega_{\epsilon} =
\phi_{\epsilon}(\Omega_0)$. The infinitesimal generator  $X =
\frac{d}{d \epsilon} \phi_{\epsilon}$ is a vector field in a
neighborhood of $\Omega_0$ which  restricts to a vector field
along  $\partial \Omega_0$;  we denote by $X_{\nu} = \dot{\rho} \nu$
its normal component. With no essential loss of generality we may
assume that $\phi_{\epsilon} |_{\partial \Omega_0}$ is a map of the
form \begin{equation} \label{rhodef}  x \in
\partial \Omega_0 \to x + \rho_{\epsilon} (x) \nu_x, \end{equation} where  $\rho_{\epsilon} \in
C^1([0, \epsilon_0], C^{\infty}(\partial \Omega_0))$, and we put
$\dot{\rho}(x) =\delta \rho\,(x): = \frac{d}{d\epsilon}{|_{\epsilon=0}}
\rho_{\epsilon} (x)$.
  An isospectral deformation is said to be trivial if
$\Omega_{\epsilon} \simeq \Omega_0$ (up to isometry) for
sufficiently small $\epsilon$. A domain $\Omega_0$ is said to be
spectrally rigid if all  isospectral deformations
$\Omega_{\epsilon}$ are trivial. The variation is called
infinitesimally spectrally rigid if $\dot{\rho} = 0$ for all
isospectral deformations.

In this article, we use the Hadamard variational formula of the
wave trace (apparently for the first time) to study spectral
rigidity problems (Theorem \ref{VARWTintro}).  Our main
application  is the infinitesimal spectral rigidity of ellipses
among $C^1$ curves of $C^{\infty}$ plane domains with the
symmetries of an ellipse. We orient the domains so that the
symmetry axes are the $x$-$y$ axes. The symmetry assumption is
then that each $\phi_{\epsilon}$ is invariant under $(x, y) \to
(\pm x, \pm y)$.

\begin{maintheo} \label{RIGID} Suppose that $\Omega_0$ is an ellipse, and that
$\Omega_{\epsilon}$ is a $C^1$ Dirichlet (or Neumann) isospectral
deformation of $\Omega_0$ through $C^{\infty}$ domains with  $\Z_2
\times \Z_2$ symmetry.  Then $X_{\nu} = 0 $ or equivalently
$\dot{\rho} = 0$.

\end{maintheo}

 As discussed in \S \ref{FLATINTRO} and \S \ref{FLAT}, Theorem \ref{RIGID}
implies that ellipses admit no isospectral deformations for which
the Taylor expansion of $\rho_{\epsilon} $ at ${\epsilon} = 0$ is
non-trivial. A function such as $e^{- \frac{1}{{\epsilon^2}}}$ for
which the Taylor series at $\epsilon = 0$ vanishes is called `flat' at
${\epsilon}= 0$.

\begin{maincor} \label{RIGIDCOR} Suppose that $\Omega_0$ is an ellipse, and that
$\epsilon \to \Omega_{\epsilon}$ is a $C^\infty$   Dirichlet (or
 Neumann) isospectral deformation through $\Z_2 \times \Z_2$
symmetric $C^{\infty}$ domains. Then $\rho_{\epsilon}$ must be
flat at $\epsilon = 0$. In particular,  there exist no non-trivial
real analytic curves $\epsilon \to \Omega_{\epsilon}$ of  $\Z_2 \times \Z_2$ symmetric
$C^{\infty}$ domains with the spectrum of an ellipse.
\end{maincor}

Spectral rigidity of the ellipse has been expected for a long time
and is a kind of model problem in inverse spectral theory.
Ellipses are  special  since their billiard flows and maps are
completely integrable. It was conjectured by G. D. Birkhoff that
the ellipse is the only convex smooth plane domain with completely
integrable billiards.  We cannot assume that  the deformed domains
$\Omega_{\epsilon}$  have this property, although the results of
\cite{Sib2,Z2} come close to showing that they do. The results are
somewhat analogous to the spectral rigidity of  flat tori or the
sphere in the Riemannian setting.

The main novel step in the proof is the Hadamard variational
formula for the wave trace (Theorem  \ref{VARWTintro}), which
holds for all smooth Euclidean domains $\Omega \subset \R^n$
satisfying standard `cleanliness' assumptions. It is of
independent interest and has applications to spectral rigidity
beyond the setting of ellipses. We therefore  present the proof in
detail.

The main advance over prior results is that the domains
$\Omega_{\epsilon}$ are allowed to be $C^{\infty}$ rather than
real analytic. Much less than $C^{\infty}$ could be assumed for
the domains $\Omega_{\epsilon}$, but we do not belabor the point.
For real analytic domains a length spectral rigidity result for
analytic domains with the symmetries of the ellipse was proved in
\cite{CdV}.  The method does not apply directly to
$\Delta$-isospectral deformations of ellipses since the length
spectrum of the ellipse may have multiplicities and the full
length spectrum might not be a $\Delta$-isospectral invariant. If
it were, then Siburg's results would imply that the marked length
spectrum is preserved \cite{Sib,Sib2,Sib3}. In \cite{Z,Z2} it is
shown that analytic domains with one symmetry are spectrally
determined if the length of the minimal bouncing ball orbit and
one iterate is a $\Delta$-isospectral invariant. The prior results
on $\Delta$-isospectral deformations that we are aware of  are
contained in the articles \cite{GM,PT,PT2} and concern
deformations of boundary conditions. To our knowledge, the only
prior results on $\Delta$-isospectral deformations  of the domain
are contained in \cite{MM}.  Marvizi-Melrose \cite{MM} introduce
new spectral invariants and prove certain rigidity results, but
they do not apparently settle the case of the ellipse (see also
\cite{A,A2} for further attempts to apply them to the ellipse). It
would be desirable to remove the symmetry assumption (to the
extent possible), but symmetry seems quite necessary for our
argument. Further discussion of prior results can be found in the
earlier arXiv posting of this article.

\subsection{\label{HDVARWT} Hadamard variation of the wave trace}

We now state a general result on the variation of the wave trace
on a domain with boundary under variations of the boundary.

To state the result, we need some notation. We denote by
\begin{equation} \label{EBSB} E_B(t) = \cos \big( t \sqrt{ -
\Delta_{B}}\;\big), \;\;\; \mbox{resp.}\;\;S_B(t) = \frac{\sin
\big(t \sqrt{- \Delta_{B }}\big)}{\sqrt{ -\Delta_{B }}}
\end{equation} the even (resp. odd) wave operators of a domain $\Omega$ with boundary conditions
$B$. We recall that $E_B(t)$ has a distribution trace as a
tempered distribution on $\R$. That is, $E_B(\hat{\rho}) =
\int_{\R} \hat{\rho}(t) E_B(t) dt$ is of trace class for any
$\hat{\rho} \in C_0^{\infty}(\R)$; we refer to \cite{GM2,PS} for
background.

The Poisson relation of a manifold with boundary gives a precise
description of the singularities of this distribution trace in
terms of periodic transversal  reflecting rays of the billiard
flow, or equivalently periodic points of the billiard map.
 For the definitions of `billiard map', `clean', `transversal
reflecting rays' etc. we refer to \cite{GM,GM2,PS}. A periodic
point of the billiard map $\beta: B^*
\partial \Omega \to B^*
\partial \Omega$ on the unit ball bundle $B^* \partial \Omega$ of the boundary
corresponds to a billiard trajectory, i.e an orbit of the billiard
flow $\Phi^t$ on $S^* \Omega$. We define the `length' of the
periodic orbit of $\beta$ to be the length of the corresponding
billiard trajectory in $S^* \Omega$.   Note that the `period' of a
periodic point of $\beta$ is ambiguous since it could refer to
this length or to the power of $\beta$.  We also denote by
$Lsp(\Omega)$ the length spectrum of $\Omega$, i.e.  the set of
lengths of closed billiard trajectories.  The perimeter of
$\Omega$ is denoted by $|\partial \Omega|$.

 In the following
deformation theorem, the boundary conditions are fixed during the
deformation and we therefore do not include them in the notation.

\begin{maintheo} \label{VARWTintro} Let $\Omega_0 \subset \R^n$ be a $C^{\infty}$
Euclidean  domain with the property that the fixed point sets of
the billiard map are clean.  Then, for any $C^1$ variation of
$\Omega_0$ through $C^{\infty}$ domains $\Omega_{\epsilon}$, the
variation of the wave traces $\delta Tr \cos\big(  t
\sqrt{-\Delta}\big)$, with  Dirichlet (or Neumann)  boundary
conditions is a classical co-normal distribution for $t \not= m
|\partial \Omega_0|$ ($m \in \Z$) with singularities contained in
$Lsp(\Omega_0)$. For each $T \in Lsp(\Omega_0)$ for which the set
$F_T$ of periodic points of the billiard map $\beta$ of
length $T$ is a $d$-dimensional clean fixed point set consisting
of transverse reflecting rays,
 there exist non-zero constants $C_{\Gamma}$
independent of $\dot{\rho}$ such that, near $T$, the leading order
singularity is
$$\delta \;Tr \; \cos\big( t \sqrt{-\Delta}\big) \sim -\frac{t}{2}\; \Re \big\{ \big( \sum_{\Gamma \subset F_T} C_{\Gamma} \int_{\Gamma} \dot{\rho}\; \gamma_1 \;  d \mu_{\Gamma}
\big) \; (t - T+ i 0)^{- 2 - \frac{d}{2}} \big\}, $$ modulo lower
order singularities. The sum is over the connected components $\Gamma$ of $F_T$. (Here
$\delta=\frac{d}{d\epsilon}|_{\epsilon=0}$. See (\ref{gamma}) for
the definition of $\gamma_1$).

\end{maintheo}

Here, the function $\gamma_1$ on $B^*
\partial \Omega$ is  defined in (\ref{gamma}) and appeared earlier
in \cite{HZ}. The densities $d\mu_{\Gamma}$ on the fixed point
sets of $\beta$ and its powers are the canonical densities defined
in Lemma 4.2 of \cite{DG}, and further discussed in
\cite{GM,PT,PT2}. The constants $C_{\Gamma}$ are explicit and
depend on the boundary conditions. We suppress the exact formulae
since we do not need them, but their definition is reviewed in the
course of the proof.

To clarify the dimensional issues, we note that there are four
closely related definitions of the set of closed billiard
trajectories (or closed orbits of the billiard map). The first is
the fixed point set of the billiard flow $\Phi^T$ at time $T$ in
$T^* \Omega$. The second is the set of unit vectors in the fixed
point set. The third is the fixed point set of the billiard flow restricted to $T^*_{\partial \Omega} \Omega$, the set of covectors with foot points at the boundary. The fourth is the set of periodic points of the billiard
map $\beta$ on $B^*
\partial \Omega$ of length $T$,  where  as above the length is defined by the length of the corresponding
billiard trajectory.  The dimension $d$ refers to the dimension of
the latter.
 In the case of the ellipse, for instance, $d = 1$; the periodic
 points of a given length form invariant curves for $\beta$.

To prove Theorem \ref{VARWTintro}, we use the Hadamard variational
formula for the Green's kernel to give an exact formula for the wave trace variation
(Lemma \ref{HDVARTR}). We then prove that it is a classical conormal distribution
 and calculate its principal symbol.

It is verified in \cite{GM} that the ellipse satisfies the
cleanliness assumptions. We then have,

\begin{maincor} \label{VARWTell}  For any $C^1$ variation of an ellipse
through  $C^{\infty}$ domains $\Omega_{\epsilon}$,   the leading
order singularity of the wave trace variation is,
$$\delta \;Tr \; \cos \big( t \sqrt{-\Delta}\big) \sim -\frac{t}{2}\; \Re \big\{ \big( \sum_{\Gamma \subset F_T}
C_{\Gamma} \int_{\Gamma} \dot{\rho}\; \gamma_1 \;  d \mu_{\Gamma}
\big) \; (t - T+ i 0)^{- \frac{5 }{2} }\big\}, $$ modulo lower order
singularities, where the sum is over the connected components $\Gamma$ of
the set $F_T$ of periodic points of $\beta$ (and its powers) of length $T$.

\end{maincor}

\subsection{\label{FLATINTRO} Flatness issues}

We now  discuss an apparently new flatness issue in isospectral
deformations. The rather technical assumption that
$\Omega_{\epsilon}$ is a $C^1$ family of $C^{\infty}$ domains
rather than a $C^{\infty}$ family in the $\epsilon$ variable is made to
deal with a somewhat neglected and obscure point about isospectral
deformations. Isospectral deformations are curves in the
`manifold' of domains. The curve might be a non-trivial  $C^{\infty}$
family in ${\epsilon}$ but the first derivative
$\dot{\rho}$ might vanish at ${\epsilon} = 0$.  Thus,
 infinitesimal spectral rigidity   is at
least apparently weaker than spectral rigidity. We impose the
$C^1$ regularity to allow us to reparameterize the family and show
that the first derivative of any $C^1$ re-parametrization must be
zero. This is not the primary focus of Theorem \ref{RIGID}, but
with no additional effort the proof extends to the $C^1$ case.

This flatness issue does not seem to have arisen before in inverse
spectral theory, even when the main conclusions are derived from
infinitesimal rigidity.  The main reason is that first order
perturbation theory very often requires analytic perturbations
(i.e. analyticity in the deformation parameter $\epsilon$), and so
most (if not all) prior results on isospectral deformations assume
that the deformation is real analytic. But our proof is based on
Hadamard's variational formula, which  is  valid for $C^1$
perturbations of domains and so we can study this more general
situation. Further, the prior spectral rigidity results (e.g.
\cite{GK}) are proved for an open set of domains and metrics and
therefore flatness at all points implies triviality of the
deformations. We are only deforming the one-parameter family of
ellipses and therefore cannot eliminate flat isospectral
deformations  by that kind of argument. We also note that there
could exist continuous
 but non-differentiable  isospectral
deformations.
\subsection{Pitfalls and complications}

The route taken in the proof of  Theorem \ref{RIGID}, and the
flatness issues just discussed, reflect certain technical issues
that arise in the inverse problem.

 First is the issue of
multiplicities in the eigenvalue spectrum or in the length
spectrum. The multiplicities of the $\Delta$-eigenvalues of the
ellipse (for either Dirichlet or Neumann boundary conditions)
appear to be almost completely unknown. If a sufficiently large
portion of the eigenvalue spectrum were simple (i.e. of
multiplicity one), one could simplify the proof of Theorem
\ref{RIGID} by working directly with the eigenfunctions and their
semi-classical limits (as in the first arXiv posting of this
article). The dual  multiplicity of the length spectrum is also
largely unknown for the ellipse. Without length spectral
simplicity one cannot work with the wave trace invariants. Our
proof relies  on the observation in \cite{GM} that the
multiplicities have to be one (modulo the symmetry) for periodic
orbits that creep close enough to the boundary.

Second is the issue of cleanliness. Theorem  \ref{VARWTintro} and
Corollary \ref{VARWTell} would  apply to any of the deformed
domains $\Omega_{\epsilon}$ if the fixed points sets were known to
be clean. One could then use the conclusion of Corollary
\ref{VARWTell}  to rule out flat isospectral deformations.
However, we do not know that the fixed point sets are clean for
the deformed domains even though we do know that they have the
same wave trace singularities as the ellipse. Equality of the wave
traces for isospectral deformations of ellipses  shows that the
periodic points of $\beta$ of $\Omega_{\epsilon}$ can never be
non-degenerate. Hence the deformations are very non-generic. It is
 plausible that equality of wave traces forces the sets of
periodic points to be clean invariant curves of dimension one. But
we do not know how to prove this kind of inverse result at this
time.

\section{\label{HASECT} Hadamard variational formula for wave traces}

In this section we  consider the Dirichlet (resp. Neumann)
eigenvalue problems for a one parameter family of smooth Euclidean
domains $\Omega_{\epsilon} \subset \R^n$,
\begin{equation} \label{EIG}  \left\{ \begin{array}{ll} - \Delta_{B_ \epsilon} \Psi_{j}(\epsilon)  =
\lambda_j^2(\epsilon) \Psi_{j}(\epsilon) \;\; \mbox{in}\;\;
\Omega_{\epsilon}, & \;\;
\\ \\ B_{\epsilon} \Psi_{j}(\epsilon)
 = 0 ,
\end{array} \right. \end{equation}
where the boundary condition $B_{\epsilon}$ could be $B_{\epsilon}
\Psi_j (\epsilon) = \Psi_j(\epsilon)|_{\partial
\Omega_{\epsilon}}$ (Dirichlet) or $
\partial_{\nu_{\epsilon}} \Psi_{j}(\epsilon)|_{\partial
\Omega_{\epsilon}}$ (Neumann). Here, $\lambda_j^2(\epsilon)$ are the
eigenvalues of $-\Delta_{B_\epsilon}$, enumerated in order and with
multiplicity, and  $\partial_{\nu_{\epsilon}}$ is the interior
unit normal to $\Omega_{\epsilon}$. We do not assume that
$\Psi_j(\epsilon)$ are smooth in $\epsilon$.   We now review the
Hadamard variational formula for the variation of Green's kernels,
and adapt the formula  to give the variation of the (regularized)
trace of the wave kernel. Our references are
\cite{G,Pee,FTY,O,FO}.

We further  denote by  $d q$ the surface measure on the boundary
$\partial \Omega$ of a domain $\Omega$, and by $\rr u =
u|_{\partial \Omega}$ the trace operator. We further denote  by
$\rr^D u =
\partial_{\nu} u |_{\partial \Omega}$ the analogous Cauchy data
trace for the Dirichlet problem. We simplify the notation for the
following boundary traces $K^b(q', q) \in \dcal'(\partial \Omega
\times \partial \Omega)$ of a Schwartz kernel $K(x, y) \in
\dcal'(\R^n\times \R^n)$ (or more precisely a distribution defined
in a neighborhood of $\partial \Omega \times \partial \Omega)$:
\begin{equation} \label{DGa} K^b(q',q)= \left\{ \begin{array}{ll}  \big( \rr_{q'} \rr_{q} N_{\nu_{q'}} N_{\nu_{q}} K\big)
(q', q), & \mbox{ Dirichlet}, \\ & \\
 (\nabla^T_{q'} \nabla_{q}^T \rr_{q'}
\rr_{q} K)(q',q) + \big(\rr_{q'} \rr_{q} \Delta_x K \big)(q',q), &
\mbox{Neumann}.
\end{array} \right.
\end{equation}
Here, the subscripts $q', q$ refer to the variable involved in the
differentiating or restricting. According to convenience, we  also
indicate this by subscripting with indices $1,2$ referring to the
first, resp. second, variable in the kernel. For instance,
$\frac{\partial}{\partial_{\nu_{q'}}} K(q',q) =
\frac{\partial}{\partial_{\nu_1}} K(q', q)$. We also use the notations
$\partial_{\nu}$ and $\frac{\partial}{\partial \nu}$
interchangeably to refer to the inward normal derivative. Also, $N_\nu$ is any smooth vector field in $\Omega$ extending $\nu$.

We are principally interested in $K(x, y) = S_B(t, x, y)$.
 In the Dirichlet, resp. Neumann,
case then we have, \begin{equation} \label{b} S_B^b(t, q', q) =
\left\{\begin{array}{l}
\rr_{q'}^D \rr_{q}^D S_D(t, q', q), \;\; \mbox{resp.}\\ \\
 \nabla_{q'}^T \nabla_{q}^T
\rr_{q'} \rr_{q} S_N(t, q', q) + \rr_{q'} \rr_{q}\; \Delta\; S_N(t, q', q).
\end{array} \right. \end{equation}

\begin{lem} \label{HDVARTR}  The variation of the wave trace with
 boundary conditions $B$ is given by,
$$\delta\; Tr\; E_B(t) = - \frac{t}{2} \int_{\partial \Omega_0} S_B^b(t, q, q) \dot{\rho}(q) dq. $$
\end{lem}
We summarize by writing,
$$ \delta \; Tr \; E_B(t)  = -\frac{t }{2} Tr_{\partial \Omega_0}
\;\dot{\rho}\; S_B^b. $$ We prove the Lemma by relating the variation
of the  wave trace to the known variational formula for the
Green's function (resolvent kernel). We now review the latter.

\subsection{Hadamard variational formulae}

 In the Dirichlet case, the classical Hadamard variational formulae states that, under a
sufficiently smooth deformation $\Omega_{\epsilon}$,
\begin{equation} \label{deltaGD} \delta G_D (\lambda, x, y) = - \int_{\partial
\Omega_0} \frac{\partial}{\partial \nu_2} G_D(\lambda, x, q)
\frac{\partial}{\partial \nu_1} G_D(\lambda, q, y) \dot{\rho}(q)
\,dq.
\end{equation} In the Neumann case,
\begin{equation} \label{deltaGN} \delta G_N (\lambda, x, y) =  \int_{\partial \Omega_0}
\nabla^T_2 G_N(\lambda, x, q) \cdot \nabla^T_1 G_N(\lambda, q, y)
\dot{\rho}(q) \,dq - \lambda^2 \int_{\partial \Omega_0}  G_N(\lambda,
x, q) G_N(\lambda, q, y) \dot{\rho}(q) \,dq. \end{equation} Above,
the subscript refers to the variable with respect to which the
derivative is taken and $\nabla^T$ denotes the  derivative with
respect to the unit tangent vector.

We briefly review the  proof of the Hadamard variational formula
to clarify the definition of $\delta G(\lambda, x, y)$ and of the
other kernels. Following \cite{Pee}, we write the inhomogeneous
problem
$$\left\{ \begin{array}{l} (-\Delta + \lambda^2) u = f, \;\;
\mbox{in}\;\; \Omega; \\ \\
u = 0 \;\;(\mbox{resp.}\;  \partial_{\nu} u = 0) \;\; \mbox{on}\;\;
\partial \Omega
\end{array} \right.
$$
in  terms of  the energy integral
 $$ \begin{array}{lll}  E(u,v) &= & \int_{\Omega} \nabla u \cdot \nabla v d x + \lambda^2 \int_{\Omega} u v dx \\ && \\ &=&
 \int_\Omega v (-\Delta + \lambda^2)  u\,dx  +
\int_{\partial\Omega} v
\partial_{\nu} u\,dq,
\end{array}$$
where $\partial_{\nu}$ is the outer unit normal. The inhomogeneous
problem is to solve
$$E(u, v) = \int_{\Omega} f v dx, $$
where $v$ is a smooth test function which vanishes to order one
(resp. 0) on $\partial \Omega$ for the Dirichlet (resp. Neumann)
problem. We denote the  energy density by
 $e(u,v)=  \nabla u \cdot \nabla v + \lambda^2 u v$.

We now vary the problems over a one-parameter family of domains.
As mentioned above, we use  one-parameter families of
diffeomorphisms $\phi_{\epsilon}$ of a neighborhood of $\Omega_0
\subset \R^n$  to define the one-parameter families
$\Omega_{\epsilon} = \phi_{\epsilon}(\Omega_0)$ of domains. We
assume $\phi_{\epsilon}$ to be a $C^1$ curve of
diffeomorphisms with $\phi_0 = id$.

The variational derivative of the solution is defined as follows:
Let $u_{\epsilon} \in H^s(\Omega_{\epsilon})$. Then
$\phi_{\epsilon}^* u_{\epsilon} \in H^s(\Omega_0)$. Put  $X =
\frac{d}{d\epsilon}|_{\epsilon=0} \phi_{\epsilon}, $  and put
$$\theta_X u = \frac{d}{d \epsilon} \phi_{\epsilon}^* u_{\epsilon} |_{\epsilon = 0}. $$
Assume that $\theta_X u \in H^s(\Omega_0)$ and that $u \in H^{s +
1}(\Omega_0)$. Then $\dot{u}$ exists and $\theta_X u = \dot{u} + X
u. $ Further, let $v $ be a test function on $\Omega_0$ and use
$\phi_{\epsilon}^{-1 *} v$ as a test function on
$\Omega_{\epsilon}$. Now re-write the boundary problems as
$$\int_{\Omega_{\epsilon}} e(u, (\phi_{\epsilon}^{-1})^* v) dx = \int_{\Omega_{\epsilon}}
f_{\epsilon} ((\phi_{\epsilon}^{-1})^* v) dx. $$ Changing
variables, one pulls back the equation to $\Omega_0$ as
$$\int_{\Omega_0} e_{\epsilon}(\phi_{\epsilon}^* u_{\epsilon}, v) \phi_{\epsilon}^* dx=
\int_{\Omega_0} (\phi_\epsilon^* f_\epsilon) v \phi_\epsilon^* dx, \;\; \mbox{where}\;\;
e_{\epsilon}(u, v) = \phi_{\epsilon}^* (e(\phi_{\epsilon}^{-1 *}
u, \phi_{\epsilon}^{-1 *} v)). $$ Assuming that $\dot{u}, \theta_X
u \in H^s(\Omega_0)$ and that $u \in H^{s + 1}(\Omega_0)$, so that
$\theta_X u = \frac{d}{d \epsilon} \phi_{\epsilon}^* u_{\epsilon}
|_{\epsilon = 0} $ exists as a limit  in $H^s(\Omega_0)$,
 we have (by the computations
of \cite{Pee} (8) and (10)) that
 \begin{equation} \label{PEETRE} \begin{array}{lll}  \int_{\Omega_0} \dot{u} (-\Delta + \lambda^2)  v\,d x & = &
 \int_{\Omega_0}\dot{f} v\,dx +\int_{\partial\Omega_0} f v \dot{\rho} \,d q +  \int_{\partial
 \Omega_0} (\nabla u \cdot \nabla v - \lambda^2 uv ) \dot{\rho} dq \\ && \\&&  + \left\{ \begin{array}{ll}
 - \lambda^2 \int_{\partial
 \Omega_0}  uv  \dot{\rho} d q  & \mbox{Dirichlet}, \\ & \\
 0& \mbox{Neumann}\end{array} \right. \end{array}. \end{equation}

To obtain (\ref{deltaGD})-(\ref{deltaGN}), at least formally, one
puts $ u_{\epsilon}(\phi_{\epsilon}(x)) = G_{B,\epsilon} (\lambda,
\phi_{\epsilon}(x), y), v(x) = G_{B, 0}(\lambda, y, x),$ and
$\phi_{\epsilon}^* f_{\epsilon} = \delta_y(x).$ Thus, $\delta
G_B(\lambda, x, y) = \frac{d}{d \epsilon}|_{\epsilon = 0} G_{B,
\epsilon}(\lambda, \phi_{\epsilon}(x), y)$. Assuming $y \in
\Omega^o$ (the interior), then $y \in \Omega_\epsilon$ for sufficiently
small $\epsilon$ and one easily verifies that (\ref{PEETRE})
implies (\ref{deltaGD})-(\ref{deltaGN}). The Green's kernel
depends on $\epsilon$ as smoothly as the coefficients of operator
$\tilde{\Delta}_{\epsilon}$ on $\Omega_0$ defined by the pulled
back energy form. Indeed, the resolvent is an analytic function of
the Laplacian.

\subsection{Proof of Lemma \ref{HDVARTR}}

Rather than the Green's function, we are interested in the
Hadamard variational formula for the wave kernels $E_B(t), S_B(t)
 $ (\ref{EBSB}), or more
precisely, for their distribution traces.  In fact, by definition
of the distribution trace, we only need the variational formula
for traces of variations $\delta \int_{\R} e^{ - i \lambda t}
\hat{\psi}(t) E_B(t) dt$  of integrals of these kernels against
test functions $\hat{\psi} \in C_0^{\infty}(\R)$, which are
simpler because the Schwartz kernels are  smooth.

We  derive the  Hadamard variational formulae for wave traces from
that of the Green's function by using the identities,
\begin{equation} i \lambda R_B(\lambda) = \int_0^{\infty} e^{-i
\lambda t} E_B(t) dt, \;\;\;\frac{d}{dt} S_B(t)  = E_B(t)
\end{equation} integrating by parts and using the finite propagation speed of
$S_B(t) $ to eliminate the boundary contributions at $t = 0,
\infty$. It follows that
\begin{equation}  R_B(\lambda) = \int_0^{\infty} e^{-i
\lambda t} \; S_B(t) dt.
\end{equation}
 As mentioned above, to obtain the variational formula for the singularity expansion
 of the wave trace,  we only need
variational formula for the smooth kernels
\begin{equation}\label{RTAU} \begin{array}{lll} \int_{\R}
\hat{\psi}(t) e^{- i \lambda t}\; E_B(t) dt = \int_{\R} i \mu R_B(
\mu) \psi (\mu-\lambda) d \mu,  \end{array} \end{equation} where
$R_B(\mu) = (- \Delta_B - \mu^2)^{-1}$ is the resolvent of $-
\Delta_B$. Here we assume that $\hat{\psi}$ is supported in $\R_+$
since in the wave trace we localize its support to the length of a
closed geodesic. In the Dirichlet case, it follows that
\begin{equation}\label{FOLLOWS}  \delta \int_{\R}
\hat{\psi}(t) e^{- i \lambda t}\; E_B(t) dt = \delta \int_{\R} i\mu
R_B( \mu) \psi (\mu-\lambda) d \mu. \end{equation} We then
derive variational formulae for wave traces. In the Dirichlet
case, it follows from (\ref{FOLLOWS}) and (\ref{deltaGD}) that $$\begin{array}{lll}
\delta \int_{\R} \hat{\psi}(t) e^{- i \lambda t}\; E_D(t) dt & = &
-i\int_{\R} \mu \psi(\mu-\lambda) \int_{\partial \Omega}
\partial_{\nu_2} G_D(\mu, x,  q) \partial_{\nu_1} G_D(\mu, q, y) \dot{\rho}(q) dq  d
\mu\\ && \\ & = &
- \int_{\R} \int_0^{\infty}  e^{- i \mu t}
\psi(\mu-\lambda) \int_{\partial \Omega}
\partial_{\nu_2} E_D(t, x,  q) \partial_{\nu_1} G_D(\mu, q, y) \dot{\rho}(q) dq  d
\mu dt \\ && \\ & = & -  \int_{\R} \int_0^{\infty} \int_0^{\infty}
e^{- i \mu (t + t')} \psi(\mu-\lambda) \int_{\partial \Omega}
\partial_{\nu_2} E_D(t, x,  q) \partial_{\nu_1} S_D(t', q, y) \dot{\rho}(q) dq  d
\mu dt dt' \\ && \\ & = & -  \int_0^{\infty} \int_0^{\infty} e^{-
i \lambda (t + t')} \hat{\psi}(t + t') \int_{\partial \Omega}
\partial_{\nu_2} E_D(t, x,  q) \partial_{\nu_1} S_D(t', q, y) \dot{\rho}(q) dq dt
dt'   \\ && \\ & = & -  \int_0^{\infty} \int_{\partial \Omega}
e^{- i \lambda \tau} \hat{\psi}(\tau) \left(\int_0^{\tau}
\partial_{\nu_2} E_D(\tau - t', x,  q) \partial_{\nu_1} S_D(t', q, y)
dt' \right) \dot{\rho}(q) d\tau dq .
\end{array}
$$
The inner integral is the same if we change the argument of $E_D$
to $t'$ and that of $S_D$ to $\tau - t'$.   We then average the
two,  set $x = y$, integrate over $\Omega$ and use the angle
addition formula for $\sin$ to obtain
\begin{equation} \delta\; Tr \int_{\R}\hat{\psi}(t) e^{- i \lambda
t}  E_D(t)d t =  -\frac{1}{2}  \int_{\R} t  \hat{\psi}(t) e^{- i
\lambda t} \int_{\partial \Omega} \partial_{\nu_1}
\partial_{\nu_2} S_D(t, q, q) \dot{\rho}(q) dq dt.
\end{equation}
This is the real part of wave trace variational formula stated in
the Lemma in the Dirichlet case, i.e. the variational formula for
$\delta\; Tr E_D(t).$ The proof in the Neumann case is similar and
left to the reader.

This concludes the proof of Lemma \ref{HDVARTR}.  To navigate the
formulae, we also give a derivation based on  the Hadamard
variational formulae for eigenvalues.
 When $\lambda_j^2(0)$ is a simple eigenvalue (i.e. of
 multiplicity one), then Hadamard's variational formula for Dirichlet eigenvalues of Euclidean domains states that
$$\delta ({\lambda}_j^2) = \int_{\partial \Omega_0} ( \partial_{\nu}
\Psi_{j}|_{\partial \Omega_0})^2  \dot{\rho}(q)\, dq, $$ where $\Psi_j$ is an $L^2$ normalized eigenfunction for the eigenvalue $\lambda_j^2(0)$ and $dq$
is the induced surface measure. See \cite{G}.  The same
comparison shows that if the eigenvalue $\{\lambda_{j}^2(0)\}_{k =
1}^{m(\lambda_{j}(0))} $  is
 multiple and if $\{\lambda_{j, k}^2(\epsilon)\}_{k = 1}^{m(\lambda_j(0))}$ is
the perturbed set of eigenvalues, then \begin{equation}
\label{HDVAR} \delta  \sum_{k = 1}^{m(\lambda_{j}(0))} \lambda_{j,
k}^2 = \sum_{k = 1}^{m(\lambda_{j}(0))} \int_{\partial \Omega_0}
(\partial_{\nu} \Psi_{j, k}|_{\partial \Omega_0})^2 \dot{\rho}(q) \,
dq= \int_{\partial \Omega_0}
\partial_{\nu_1} \partial_{\nu_2} \Pi_{\lambda_j(0)}(q,q)\, \dot{\rho}(q) dq,
\end{equation}
where $\{\Psi_{j,k}\}_{k=1}^{m(\lambda_j(0))}$ is an orthonormal
basis for the eigenspace of the multiple eigenvalue
$\lambda_j^2(0)$ and $\Pi_{\lambda_j(0)}(x, y)$ is its spectral
projections kernel. Since  $\delta \lambda_j = \frac{\delta
(\lambda_j^2)}{2 \lambda_j},$ by (\ref{HDVAR}) we have
$$ \begin{array}{lll} \delta \; Tr \; E_B(t)= \delta\; \sum \cos(t\lambda_{j,k})= -t \sum_j  \big(\sum_{k=1}^{m(\lambda_j(0)} \delta (\lambda_{j,k}^2)\big) \frac{\sin(t\lambda_{j}(0))}{2\lambda_{j}(0)}=
 \frac{-t}{2} \int_{\partial \Omega_0}
\partial_{\nu_1} \partial_{\nu_2} S_B(t, q, q) \dot{\rho}(q) dq.
\end{array}$$  Hence Lemma \ref{HDVARTR} follows in the Dirichlet case.

There exist similar Hadamard variational formulae in the Neumann
case. When the eigenvalue is simple, we have
\begin{equation} \label{HDVARN} \delta ({\lambda}_j^2) = \int_{\partial \Omega_0}
\left(|\nabla^T_q (\Psi_{j} |_{\partial \Omega_0}(q))|^2 -
\lambda_j^2(0) (\Psi_{j}|_{\partial \Omega_0}(q))^2 \right)
\dot{\rho}(q) \, dq,
\end{equation}
For a multiple eigenvalue we sum over the expressions over an
orthonormal basis of the eigenspace. The result does not depend on
a choice of orthonormal basis. Similar computation using (\ref{HDVARN}) follows to show Lemma \ref{HDVARTR} for the Neumann case.

\section{Proof of  Theorem  \ref{VARWTintro}}

We now study the singularity expansion of $\delta Tr \cos (t
\sqrt{-\Delta_B})$ and prove Theorem \ref{VARWTintro}. At
first sight, one could do this in two ways: by taking the
variation of the spectral side of the formula, or by taking the
variation of the singularity expansion. It seems simpler and
clearer to do the former since we do not know how the invariant
tori of the integrable elliptical billiard deform under an
isospectral deformation. In this section we will drop the
subscript $0$ in $\Omega_0$.

The variational formula for $\delta Tr \cos( t
\sqrt{-\Delta_B})$  is given in Lemma \ref{HDVARTR}. In the Dirichlet case, by  (\ref{DGa})-(\ref{b}),
\begin{equation} \label{DG} Tr_{\partial \Omega}\; \dot{\rho}\;
S_D^b = \pi_* \;  \Delta^* \dot{\rho}\;\big(  \rr_1 \rr_2
N_{\nu_1} N_{\nu_2}S_D(t, x, y)\big),
\end{equation} where $N_{\nu}$ is any smooth vector field in
$\Omega$ extending $\nu$, and where the subscripts indicate the
variables on which the operator acts. In the Neumann case by
(\ref{DGa})-(\ref{b}),
\begin{equation} \label{NG} Tr_{\partial \Omega}\; \dot{\rho}\;
S_N^b = \pi_*\; \Delta^* \dot{\rho} \; \left((\nabla^T_1 \nabla_2^T
\rr_1 \rr_2 - \rr_1 \rr_2\;\Delta_x) S_N(t, x, y) \right).
\end{equation} Here,  $\Delta:
\partial \Omega \to \partial \Omega \times \partial \Omega$ is the
diagonal embedding $q \to (q,q)$ and $\pi_*$ (the pushforward of
the natural projection $\pi: \partial \Omega \times \R \to \R$) is
the integration over the fibers with respect to the surface measure $dq$.
The duplication in  notation between the Laplacian and the
diagonal is regrettable, but both are standard and should not
cause confusion. Since $S_B(t, x, y)$ is microlocally a Fourier
integral operator near the transversal periodic reflecting rays of
$F_T$, it will follow from (\ref{DG}) that the trace is locally a
Fourier integral distribution near $t = T$.

We are assuming that the set of periodic points of the billiard
map corresponding to space-time billiard trajectories of length $T
\in Lsp(\Omega)$ is a submanifold $F_T$ of $B^*
\partial \Omega$.
 We thus fix $T \in Lsp(\Omega)$ consisting only
of periodic reflecting rays, i.e. we assume $T \not= m |\partial
\Omega|$ ($|\partial \Omega|$ being the perimeter)  for $m \in
\Z$. In order to study the singularity of the  boundary trace
 near a component $F_T$ of the fixed point set, we construct a
pseudo-differential cutoff $\chi_T = \chi_{T}(t, D_t, q, D_q) \in
\Psi^0_{}(\R \times \partial \Omega)$ whose complete symbol
$\chi_{T}(t, \tau, q, \zeta)$ has the form $\chi_T(q,
\frac{\zeta}{\tau})$ with $\chi_T(y, \zeta)$ supported in a small
neighborhood of the fixed point set  $F_T \subset B^*
\partial \Omega$, equals one in a smaller neighborhood, and in particular vanishes in a neighborhood of the glancing directions in  $S^*
\partial \Omega =
\partial (B^* \partial \Omega)$. Since the symbol of $\chi_T$ is independent of $t$ we will instead use $\chi_{T}(D_t, q, D_q)$. We may assume that the support of
the cutoff is invariant under the billiard map $\beta$. Therefore we need to study the operator
\begin{equation}\label{localizedtrace}\pi_* \Delta^* \;\dot{\rho} \;\chi_{T}(D_t, q', D_{q'}) \chi_{T}(D_t,q, D_{q}) S^b_B(t,q', q),
\end{equation} and compute its symbol. To do this we first study the operators $r$ and $S_B(t)$ and review their basic properties. Next we study the composition $$\chi_{T}(D_t, q', D_{q'}) \chi_{T}(D_t,q, D_{q}) S^b_B(t,q' ,q)$$ and compute its symbol. Finally in Lemma \ref{symboloftrace} we take composition with $\pi_* \Delta^* \; \dot{\rho}$ and calculate the symbol of (\ref{localizedtrace}).

\subsection{FIOs and their symbol}
We recall that the principal symbol $\sigma_I$ of a
Fourier integral distribution
$$I = \int_{\R^N} e^{i \phi (x, \theta) }
a(x, \theta) d \theta,\qquad \quad I\in I^m(M, \Lambda_{\phi}), $$ of order $m$ is defined in terms of the
parametrization
 $$\iota_{\phi}: C_{\phi} = \{(x, \theta): d_{\theta} \phi = 0\}
 \to (x, d_x \phi) \in \Lambda_{\phi} \subset T^* M$$
 of the associated Lagrangian $\Lambda_{\phi}$. It is a half density on $\Lambda_{\phi}$ given by $\sigma_I={(\iota_\phi)}_*(a_0 |d_{C_\phi}|^\half ) $  where $a_0$ is the leading term of the classical symbol
 $a \in S^{m+\frac{n}{4}-\frac{N}{2}}(M \times \mathbb R^N)$, $n=$dim$M$ and
$$d_{C_{\phi}}: = \frac{d c}{|D(c,
\phi_{\theta}')/D(x, \theta)|} $$ is the Gelfand-Leray form on
$C_{\phi}$ where $c$ is a system of coordinates on
$C_{\phi}$. For notation and background we refer to \cite{Ho}. When $I(x,y) \in I^m(X\times Y, \Lambda)$ is the kernel of an FIO it is very standard to use the symplectic form $\omega_X - \omega_Y$ on $X\times Y$ and define
$$\iota_{\phi}: C_{\phi} = \{(x,y, \theta): d_{\theta} \phi = 0\}
 \to (x, d_x \phi, y, -d_y \phi) \in \Lambda_{\phi} \subset T^* X \times T^*Y.$$ We will call $\Lambda_\phi$ the canonical relation of $I(x,y)$.

\subsection{The restriction operator $r$ as an FIO}

The restriction $\rr$ to the boundary satisfies,
 $\rr \in I^{ \frac{1}{4}}(\partial
\Omega \times \R^n, \Gamma_{\partial \Omega}), $  with the canonical relation
\begin{equation} \label{LAMBDADELTA}\Gamma_{\partial
\Omega}= \{(q, \zeta, q, \xi) \in T^*
\partial \Omega \times T^*_{\partial \Omega} \R^n; \xi|_{T_q
\partial \Omega} = \zeta\}. \end{equation}  The adjoint then satisfies
 $\rr^{ *}\in I^{ \frac{1}{4}}(\R^n \times \partial
\Omega, \Gamma_{\partial \Omega}^*)$, where
$$ \Gamma_{\partial \Omega}^* = \{(q, \xi, q, \zeta) \in
T^*_{\partial \Omega} \R^n \times T^*
\partial \Omega; \xi|_{T_q
\partial \Omega} = \zeta\}.$$
Here, $T^*_{\partial \Omega} \R^n$ is the set of co-vectors to
$\R^n$ with footpoint on $\partial \Omega$. We parameterize
$\Gamma_{\partial \Omega}$ (\ref{LAMBDADELTA}) by $T^{*
+}_{\partial \Omega}(\Omega)$, the inward pointing covectors,
using the Lagrange immersion
\begin{equation} \label{PARAMGAM} \iota_{\Gamma_{\partial \Omega}} (q, \xi) = (q, \xi|_{T_q(\partial \Omega)}, q, \xi). \end{equation}
To prove these statements, we introduce
 Fermi normal coordinates $(q, x_n)$  along
$\partial \Omega$, i.e.  $x = \exp_{q}( x_n \nu_{q})$ where $\nu_q$
is the interior unit normal at $q$.  Let
  $\xi = (\zeta, \xi_n) \in T_{(q,x_n)}^*\R^n$ denote the corresponding symplectically dual fiber
 coordinates. In these coordinates, the kernel of $r$ is given by
  \begin{equation} \label{OPAH}   \rr (q, (q',x'_n)) = C_n \int_{\R^n} e^{i \langle q -q', \zeta \rangle  - i x'_n \xi_n}  d\xi_n
  d\zeta.
  \end{equation}
  The phase $\phi(q, ( q',x'_n), (\zeta, \xi_n)) = \langle q - q', \zeta
   \rangle - x'_n \xi_n $ is non-degenerate and its critical set
   is   $C_{\phi} =\{(q, q', x'_n, \xi_n, \zeta); q' = q, x'_n = 0, \} $. The  Lagrange map
   $\iota_{\phi}:(q, q, 0, \xi_n, \zeta) \to (q, \zeta, q, \zeta, \xi_n) $ embeds $C_{\phi} \to T^*\partial \Omega \times T^*\R^n$ and maps onto $\Gamma_{\partial \Omega}$.
The adjoint kernel has the form $K^*(x, q) = \bar{K}(q, x)$ and
therefore has a similar oscillatory integral representation. It
is clear from ($\ref{OPAH}$) that the order of $r$ as an FIO is
$\frac{1}{4}$. Also, in the parametrization (\ref{PARAMGAM}), the
principal symbol of $r$ is $\sigma_r=|dq \wedge d \zeta \wedge d
\xi_n|^{\half}$.

\subsection{Background on parametrices for $S_B(t)$}
We first review the Fourier integral description of $E_B(t)$,
$S_B(t)$ microlocally near transversal reflecting rays. This is
partly for the sake of completeness, but  mainly because we need
to compute their principal symbols (and related ones) along the
boundary. Although the principal symbols are calculated in the
interior in \cite{GM2,PS} and elsewhere (see  Proposition 5.1 of
\cite{GM2}, section 6 of \cite{MM}  and section 6 of
\cite{PS}), the results do not seem to be stated along the
boundary (i.e. the symbols are not calculated at the boundary). The statements we need are contained in Theorem 3.1 of
\cite{Ch2} (and its proof), and we largely follow its
presentation.

We need to calculate the canonical relation and principal symbol
of the wave group, its derivatives and their restrictions to the
boundary. We begin by recalling that the propagation of
singularities theorem  for the mixed Cauchy-Dirichlet (or Neumann)
problem for the wave equation states that the wave front set of
the wave kernel satisfies,
$$WF (S_B(t, x, y)) \subset \bigcup_{\pm} \Lambda_{\pm}, $$ where $\Lambda_{\pm} = \{(t, \tau,x, \xi, y, \eta):\; (x, \xi)= \Phi^t(y, \eta), \;\tau =\pm |\eta|_y\} \subset T^*(\mathbb R \times \Omega \times \Omega)$ is the graph of the generalized (broken) geodesic flow, i.e. the billiard flow $\Phi^t$. For
background we refer to \cite{GM2,PS,Ch2} and to \cite{Ho} (Vol.
III, Theorem 23.1.4 and Vol. IV, Proposition 29.3.2). For the
application to spectral rigidity, we only need a microlocal
description of wave kernels away from the glancing set, i.e. in
the hyperbolic set  microlocally near periodic transversal
reflecting rays. In these regions,  there exists a microlocal
parametrix due to Chazarain \cite{Ch2}, which is more fully
analyzed in \cite{GM2,PS} and applied to the ellipse in \cite{GM}.

The microlocal parametrices for $E_B $ and $S_B$ are constructed
in the ambient space $\R \times \R^n \times \Omega$. Since $E_B =
\frac{d}{dt} S_B$ it suffices to consider the latter. Then there
exists a Fourier integral (Lagrangian) distribution,
$$\tilde{S}_B(t, x, y) = \sum_{j=-\infty}^\infty S_j(t, x, y), \;\;\mbox{with}\;\;
S_j\in I^{-\frac{1}{4} - 1} (\R \times \R^n \times \Omega,
\Gamma_\pm^j)$$ which microlocally approximates $S_B(t,x,y)$ modulo a
smooth kernel near a transversal reflecting ray. The sum is locally finite hence well-defined. The canonical
relation of $\tilde{S}_B$  is contained in a union
$$\Gamma = \bigcup_{\pm,\; j \in \Z} \Gamma^j_{\pm} \subset T^*(\R \times \R^n \times
\Omega)$$ of canonical relations  $\Gamma^j_{\pm}$ corresponding
to the graph of the broken geodesic flow with $j$ reflections.
Notice we let $j\in \mathbb Z$ which is different from \cite{Ch2} where $j$ goes from $0$ to $\infty$ and where the two graphs $\Gamma_\pm^j$ and $\Gamma_\pm^{-j}$ are combined.

To describe $\Gamma^j_{\pm}$, we introduce some useful notation
from \cite{Ch2} with a slight adjustment. We have two Hamiltonian
flows $g^{\pm t}$ corresponding to the Hamiltonians $\pm|\eta|$.
For $(y,\eta)$ in $T^* \Omega$ we define the first impact times
with the boundary,
$$\left\{ \begin{array}{l} t_\pm^{1}(y,
\eta) = \inf\{t
> 0: \pi g^{\pm t}(y, \eta) \in \partial \Omega\}, \\ \\ t_\pm^{-1}(y, \eta) = \sup\{t < 0: \pi g^{\pm t}(y, \eta) \in \partial
\Omega\}. \end{array} \right.$$ The impact times are related by
 $t_\pm^{-1}=-t_\mp^{1}$. We define $t_\pm^j$ inductively for
$j>0$ res. $j<0$ to be the time of $j$-th reflection (i.e. impact
with the boundary) for the flow $g^{\pm t}$ as $t$ increases resp.
decreases from $t = 0$. Then we put
$$\left\{ \begin{array}{l} \lambda^1_\pm(y,\eta)=g^{\pm t^1_\pm(y,\eta)}(y,\eta) \in
T_{\partial \Omega}^* \Omega, \\ \\
\lambda^{-1}_\pm(y,\eta)=g^{\pm t^{-1}_\pm(y,\eta)}(y,\eta) \in
T_{\partial \Omega}^* \Omega. \end{array} \right. $$ Next we
define $\widehat{\lambda^1_\pm(y,\eta)}$ to be the reflection of
$\lambda^1_\pm(y,\eta)$ at the boundary. For any $(q,\xi) \in
T^*_q \R^n, q \in \partial \Omega$, the reflection  $\xi \to
\widehat{\xi}$ has the same tangential projection as $\xi$ but
opposite normal component.  Similarly we define
$\widehat{\lambda^{-1}_\pm(y,\eta)}$. Flowing
$\widehat{\lambda^1_\pm(y,\eta)}$ (resp.
$\widehat{\lambda^{-1}_\pm(y,\eta)}$) by $g^{\pm t}$ as $t$
increases (resp. decreases) and continuing the same procedure we
get
 $t^j_\pm(y,\eta)$ and $\lambda^j_\pm(y,\eta)$ for all $j\in \mathbb Z$. We also set $T^j_{\pm}=\sum_{k=1}^j t^k_{\pm}$ for $j>0$ and $T^j_{\pm}=\sum_{k=-1}^j t^k_{\pm}$ for $j<0$.

The canonical graph $\Gamma^j_{\pm}$ can now be written as
\begin{equation} \label{Gammaj}\Gamma^j_{\pm}=\left\{ \begin{array}{ll} \{(t, \tau, g^{\pm t}(y, \eta), y, \eta):\; \tau=\pm|\eta|_y\}
\qquad \qquad \quad \quad  j=0, \\ \\
\{(t,\tau, g^{\pm(t-T^j_\pm(y,\eta))}
\widehat{\lambda^j_\pm(y,\eta)}, y, \eta): \tau=\pm|\eta|_y \}
\qquad j\in \mathbb Z, j \neq 0.\end{array} \right. \end{equation}
For  each $j \in \mathbb Z$, $\bigcup_{\pm} \Gamma^j_{\pm}$ is the
union of two canonical graphs, which we refer to as its `branches'
or `components' (see figure $3.2$ of \cite{GM2} for an
illustration). These two branches arise because $S_B(t) =
\frac{1}{2i\sqrt{-\Delta_B}}( e^{i t \sqrt{-\Delta_B}} - e^{-i t
\sqrt{-\Delta_B}})$ is the sum of two Fourier integral operators
whose canonical relations are respectively the graphs of the
forward/backward broken geodesic flow and which correspond to the
two halves $\tau
> 0, \tau < 0$ of the characteristic variety $\tau^2 - |\eta|^2 =
0$ of the wave operator.

At the boundary, we  have four modes of propagation: in addition
to the two $\pm$ branches corresponding to $\tau>0$ and $\tau<0$,
there are two modes of propagation corresponding to the two
`sides' of $\partial \Omega$. To illustrate this we first discuss
the simple model of the upper half space.

\subsubsection{ Upper half space; a local model for one reflection}
Let $\R_+^n =\{(x', x_n) \in \R^{n-1} \times \R:  x_n \geq 0\}$ be
the upper half space. Denote by $S_0(t, x, y)$ the kernel of
$\frac{\sin (t \sqrt{-\Delta})}{\sqrt{-\Delta}}$ of Euclidean
$\R^n$. By the classical method of images,
$$\left\{ \begin{array}{l} S_D(t, x, y) = S_0(t, x, y) - S_0(t, x, y^*), \\ \\
S_N(t, x, y) =   S_0(t, x, y) + S_0(t, x, y^*) \end{array} \right.
$$ where $y^* \in \R^n_-$ is the reflection of $y$ through the boundary
$\R^{n-1} \times \{0\}$.

% Indeed, $y \to y^*$ is an isometry, so
%both kernels satisfy $\Box E_B = 0$ (in either the $x$ or $y$
%variable) and have the correct initial conditions since $y^* \notin
%\R_+^n$. Further they satisfy the correct boundary conditions: it
%is clear that $S_D(t, x, y) = 0$ if $y \in \R_+^{n-1} \times
%\{0\}$ since $y^* = y$ for such points. Also, if $x_n = 0$ then
%$S_D(t, x, y) = 0$ since $S_D(t, x, y) $ is a function of the
%distance $|x - y|$ and $|x - y| = |x - y^*|$ if $x_n = 0$.
%Similarly,  the normal derivative is $\frac{\partial}{\partial
%y_n}$, so the normal derivatives cancel for $S_N(t, x, y)$ when
%$y_n = 0$.
% Also, $S_0(t, x, y*) = S_0(t, x^*, y)$ and $S_0(t, x,
%y) = S_0(t, y, x)$, so the same calculation applies in the $x$
%variable.

The canonical relation associated to $S_N$ and $S_D$ is the union
of the canonical relations of $S_0$ and of $S_0^* = S_0(t, x,
y^*)$. More precisely by our notation in (\ref{Gammaj})
$$WF(S_B(t,x,y)) \subset \Gamma^{0}_\pm \cup \Gamma^1_\pm \cup \Gamma^{-1}_\pm.$$ Note that this example is
asymmetric in past and future: the forward trajectory may
intersect boundary, but then backward one does not. Also, in this example for $j>1$ and $j<-1$ the graphs $\Gamma^j_\pm$ are empty.

\subsubsection{Symbol of $S_B(t,x,y)$ in the interior}
 In the boundaryless case of \cite{DG}, the half density symbol of
$e^{i t \sqrt{-\Delta_g}}$ is a constant multiple (Maslov factor)
of the canonical graph volume half density $\sigma_{can} = |dt
\wedge dy \wedge d \eta |^{\frac{1}{2}}$
 on $\Gamma_+$ in the graph parametrization $(t, y, \eta) \to \Gamma_+=
(t, |\eta|_g, g^t(y, \eta), y, \eta)$. In the boundary case for
$E_B(t)$ the symbol in the interior is computed in Corollary 4.3
of \cite{GM2} as a scalar multiple of the graph half-density. It
is a constant multiple of the graph half-density \begin{equation}
\label{GRAPH} \sigma_{can, \pm}=|dt \wedge dy \wedge d
\eta|^\half\end{equation} in the obvious graph
 parametrization of $\Gamma^j_\pm$ in (\ref{Gammaj}); the constant
  equals $\half$ in the Neumann case and $\frac{(-1)^j}{2}$ in
   the Dirichlet case. However in \cite{GM2} the symbols are not calculated at the boundary.

\subsubsection{Symbol of $S_B(t,x,y)$ at the boundary} Since we want to restrict kernels and symbols to the
boundary, we introduce further notation for the subset of the
canonical relations lying over boundary points. Following
\cite{Ch2}, we denote by $A_{\pm}^0 = \{(0, \tau, y, \eta, y,
\eta): \tau = \pm |\eta|_y\}$ the subset of $\Gamma_{\pm}^0$ with
$t = 0$. Under the flow $\psi_{\pm}^t$  of the Hamiltonian $\tau \pm |\xi|_x$
on $\R \times \R^n$, it flows out to the graph $\Gamma^0_{\pm}$
 (denoted $C^0_{\pm}$ in \cite{Ch2}, (2.11)). One then defines $A_{\pm}^1 \subset
\Gamma_{\pm}^0$ resp. $A_{\pm}^{-1} \subset \Gamma_{\pm}^0$ as the
subset lying over $\R_+ \times \partial \Omega \times \Omega$
resp.  $\R_- \times \partial \Omega \times \Omega$. We then have
$$ \Gamma_{\pm}^1 = \bigcup_{t \in \R} \psi_{\pm}^t
\widehat{A}^1_{\pm}, $$and
$$\Gamma_{\pm}^{-1} = \bigcup_{t \in \R} \psi_{\pm}^t
\widehat{A}^{-1}_{\pm}, $$ as the flow out under the Euclidean
space-time geodesic flow of $\widehat{A}^1_{\pm}$ and $\widehat{A}^{-1}_{\pm}$. Thus, along the boundary, for $t>0$ (resp. $t<0$) $A^1_{\pm}$ and
$\widehat{A}^1_{\pm}$ (resp. $A^{-1}_{\pm}$ and
$\widehat{A}^{-1}_{\pm}$) both lie in the canonical relation of
$E_B(t), S_B(t)$. In a similar way one defines $A^2_{\pm}$ to be
the subset of $\Gamma^1_{\pm}$ lying over $\R_+ \times \partial
\Omega \times \Omega$ and $\widehat{A}^2_{\pm}$ to be its
reflection. Then also $A^2_{\pm} \cup \widehat{A}^2_{\pm}$ lies in
the canonical relation. Similarly one defines $A_{\pm}^j$
and $\widehat{A}_{\pm}^j$ for all $j \in \mathbb Z$.

\begin{rem}Since we are interested in the singularity of the trace at $t=T>0$ we will only consider the graphs $\Gamma^j_\pm$ for $j \geq 0$. Regardless of this, because $\delta \; Tr E_B(t)$ is even in $t$ it has the same singularity at in $t=L$ and $t=-T$.
\end{rem}

The symbols of $E_B(t)$ and $S_B(t)$ are half-densities on the
associated canonical relations, and therefore are sums of four
terms at boundary points, i.e. there is a contribution from each
of $A_{\pm}^j$ and $\widehat{A}_{\pm}^j$. In the interior, there
is only a contribution from the $\pm$ components.

The following Lemma gives formulas for the principal symbol of $S_B$ (and therefore $E_B$) on $\Gamma^j_{\pm}$ and its restriction to $\Gamma_{\partial \Omega} \circ (A^j_\pm  \cup \widehat{A^j_\pm})$.

\begin{lem}\label{BOUNDARYSYMBOL} Let $e_{\pm}$ be the principal symbol of $\tilde{S}_B$
when restricted to $ \Gamma_\pm=\bigcup_{j} \Gamma^j_\pm$. Let $\sigma_r$ be the principal symbol of the boundary restriction operator $r$.
Then\begin{itemize}
\item[1.] In the interior, on $\Gamma_\pm^j$, up to Maslov factors we have  $e_\pm=\frac{(-1)^j}{2\tau}\sigma_{can,\pm}=\pm\frac{(-1)^j}{2|\eta|}
 \sigma_{can, \pm}$ in the Dirichlet case, and $e_\pm=\frac{1}{2\tau}\sigma_{can,\pm}=\pm\frac{1}{2|\eta|}
 \sigma_{can, \pm}$ in the Neumann case.
\item[2.] At the boundary, on $\Gamma_{\partial \Omega} \circ A^j_\pm =\Gamma_{\partial \Omega} \circ \widehat{A^j_\pm}$ we have

\noindent In the Dirichlet case:
$$\sigma_r \circ
e_{\pm}(t_{\pm}^j, \pm \tau, \widehat{\lambda_{\pm}^j(y, \eta)},
y, \eta) = -  \sigma_r \circ e_{\pm}(t^j_{\pm}, \pm \tau,
\lambda_{\pm}^j(y, \eta), y, \eta), $$
\noindent In the Neumann case:
$$ \sigma_r \circ
e_{\pm}(t_{\pm}^j, \pm \tau, \widehat{\lambda_{\pm}^j(y, \eta)},
y, \eta) = \sigma_r \circ e_{\pm}(t^j_{\pm}, \pm \tau,
\lambda_{\pm}^j(y, \eta), y, \eta).$$

\end{itemize}
\end{lem}

\begin{proof} These formulas are obtained from the transport
equations in \cite{Ch2}, $(b_0') - (e_0') $ (page 175). We now sketch the proof.

The transport equations for the symbols of $E_B, S_B$ determine how they
propagate along broken geodesics.  As in the boundaryless case,
the principal symbol has a zero Lie derivative, $\lcal_{H_{\tau+|\xi|}}
\sigma_E = 0$,  in the interior along geodesics. The important
point for us is the rule by which they are reflected at the
boundary. Let $\sigma_B$ be the principal symbol of the boundary restriction operator $B$ defined in (\ref{EIG}) ($B=r$ resp. $B=rN$ when we have Dirichlet resp. Neumann boundary condition) and let $\sigma_0$ be the principal symbol of the restriction operator to $t=0$. Then,
\begin{equation} \label{TRANSPORT} \left\{ \begin{array}{llll}
(b_0): & (\frac{d^2}{dt^2}-\Delta_B) \tilde S_B \sim  0  \Longrightarrow    (b_0'): \; \lcal_{\psi^t_\pm} e_\pm =0\\ & \\
(c_0): & \tilde{S_B}|_{t=0} \sim 0 \Longrightarrow (c_0'):\; \sigma_0 \circ e_+(0, \tau, y, \eta, y, \eta)+\sigma_0\circ e_-(0, -\tau, y, \eta, y, \eta)=0\\ & \\
(d_0): & \frac{d}{dt}|_{t=0} \tilde{S_B} \sim \delta(x-y) \Longrightarrow (d_0'): \; \tau \big( \sigma_0 \circ e_+(0, \tau, y, \eta, y, \eta)-\sigma_0\circ e_-(0, -\tau, y, \eta, y, \eta)\big)=\sigma_I\\ &\\
(e_0): &  B \tilde S_B \sim 0 \Longrightarrow  (e_0'): \;\; \sigma_B \circ e_\pm=\sigma_B \circ (e_\pm|_{A^j_\pm})+ \sigma_B \circ( e_\pm|_{\widehat{A^j_\pm}})=0.
\end{array} \right. \end{equation}
Here $\sigma_I$ is the principal symbol of the identity operator. The implication  $(b_0) \Longrightarrow (b_0')$ follows for example from Theorem 5.3.1 of \cite{DH}. The other implications are obvious. From $(c_0')$ and $(d_0')$ we get
$$(\sigma_0 \circ e_\pm) (y,\eta, y,\eta)=\frac{1}{2\tau}\sigma_I, \qquad  \text{on} \quad T^*\Omega.$$
But by $(b_0')$, the symbol $e_\pm$ is invariant under the flow
$\psi_\pm^t$ and therefore the first part of the Lemma follows but
only on $\Gamma_\pm^0$. The second part of the Lemma follows from
$(e_0')$.  The first term of $(e_0')$ is known from the previous
transport equations. Hence $(e_0')$ determines the `reflected
symbol' at the $j$-th impact time and impact point. In the
Dirichlet case, $B$ is just $r$ the restriction to the boundary
and so the reflected principal symbol is simply the opposite of
the direct principal symbol. In the Neumann case, $B$ is the
product of the symbol  $\langle \lambda_{\pm}^1(y, \eta), \nu_y
\rangle$ of the inward normal derivative times restriction $r$.
The reflected symbol thus equals the direct symbol since the sign
is canceled by the sign of the $\langle
\widehat{\lambda_{\pm}^1(y, \eta)}, \nu_y \rangle=-\langle
\lambda_{\pm}^1(y, \eta), \nu_y \rangle$ factor. Thus, the volume
half-density is propagated unchanged in the Neumann case and has a
sign change at each impact point in the Dirichlet case. It follows
that, on $\Gamma_\pm^j$ and after $j$ reflections, the Dirichlet
wave group symbol is $(-1)^j$ times $\frac{1}{2\tau}$ times the
graph half-density \eqref{GRAPH} and the Neumann symbol is
$\frac{1}{2\tau}$ times the graph half-density.

\end{proof}

\subsection{ $\chi_{T}(D_t, q', D_{q'}) \chi_{T}( D_t,q, D_{q}) S^b_B(t,q',q) $ is a Fourier integral operator}

\begin{lem} \label{partial} We have, $$\chi_{T}(D_t, q', D_{q'}) \chi_{T}(D_t, q, D_{q})S^b_B(t, q', q) \in I^{\half + 1 - \frac{1}{4}} (\R
\times
\partial \Omega \times \partial \Omega, \Gamma_{\partial,\pm}).
$$
 Here,
$\Gamma_{\partial,\pm} = \bigcup_{j \in \mathbb Z}\Gamma^j_{\partial,\pm}, $ with
$$\begin{array}{lll} \Gamma^j_{\partial,\pm} :&= & \{(t, \tau, q', \zeta', q, \zeta) \in T^* (\R
\times \partial \Omega \times \partial \Omega): \exists \;\xi' \in
T^*_{q'} \R^n, \xi \in T^*_{q} \R^n: \\ &&\\ &&  (t, \tau, q', \xi',
q, \xi) \in \Gamma^j_\pm, \;\; \xi'|_{T_{q'}
\partial \Omega} = \zeta',\; \xi|_{T_{q}
\partial \Omega} = \zeta  \}.
\end{array}$$

\end{lem}

\begin{proof} We only show the proof in the Dirichlet case. The Neumann case is very similar. The kernel $\chi_{T}(D_t, q', D_{q'}) \chi_{T}(D_t, q, D_{q})S^b_D(t, q', q)$ for fixed $t$ is the Schwartz kernel of
the composition
\begin{equation} \label{COMP} \chi_{T} \circ (\rr\; N) \circ
S_D(t)\circ (N^*\; \rr^*)\circ \chi_{T}^* : L^2 (\partial
\Omega)\to L^2 (\partial \Omega),
\end{equation} where $\rr^*$ is the adjoint of
$\rr: H^{\half} (\bar{\Omega}) \to L^2(
\partial \Omega)$.

To prove the Lemma, we use that   $\rr$ is a Fourier integral
operator with a folding canonical relation, and that the
composition (\ref{COMP}) is transversal away from the tangential
directions to $\partial \Omega$, where $S_B(t)$ fails to be a
Fourier integral operator. The cutoff $\chi_T$ removes the part of
the canonical relation near the fold locus, hence the composition
is a standard Fourier integral operator.

By the results cited above in  \cite{Ch2,GM2,PS,MM,Ch2},
microlocally away from the gliding directions, the wave operator $
S_B(t) $ is a Fourier integral operator associated to the
canonical relations $\Gamma^j_{\pm}$.  Since $\Gamma^j_{\pm}$ is a
union of  graphs of  canonical transformations, its composition
with the canonical relation of $\rr^D=rN$ is automatically
transversal. The further composition with the canonical relation
of $\rr^{D *}$ is also transversal. Hence, the composition is a
Fourier integral operator with the composed wave front relation
and the orders add. Taking into account that we have two boundary
derivatives, we need to add $\half$ to the order.

To determine the composite relation, we note that
\begin{equation} \label{PHI}
\begin{array}{l}\Phi_\pm: \R \times T^*_{\partial \Omega} \R^n \to T^* \R \times  T^*
\Omega \times T^*_{\partial \Omega} \Omega, \\ \\ \Phi_\pm(t, q, \zeta, \xi_n) := (t, \pm |\zeta
+ \xi_n|, \Phi^t(q, \zeta, \xi_n), q, \zeta, \xi_n) \end{array}
\end{equation} parameterizes the graph of the (space-time) billiard
flow with initial condition on $T^*_{\partial \Omega} \mathbb R^n$. Here,
$\zeta \in T^* \partial \Omega$ and $\xi_n \in N_+^* \partial
\Omega$, the inward pointing (co-)normal bundle.  $\Phi_\pm $ is a
homogeneous folding map with folds along $\R \times T^*
\partial \Omega$ (see e.g. \cite{Ho} (volume III) for background).
It follows that  $S_D(t) \circ (N^* r^*) \chi_T^*$  is a Fourier integral operator of order one associated to the canonical relation
$$\{(t, \pm |\xi|, \Phi^t(q, \xi), q, \xi|_{T^* \partial\Omega}\} \subset T^*(\R
\times \Omega \times \partial\Omega),$$ and is a local
canonical graph away from the fold singularity along $T^* \partial
\Omega$. Composing on the left by the restriction
relation  produces a Fourier integral operator with  the stated
canonical relation. The two normal derivatives $N$ of course do not change the
relation.

\end{proof}

\subsection{Symbol of  $\chi_{T}(D_t,q', D_{q'}) \chi_{T}(D_t,  q, D_q) S^b_B(t, q', q) $ }

The next step is to compute the principal symbols of the operators
in Lemma \ref{partial}.

To state the result, we need some further notation. We denote
points of $T^*_{\partial \Omega} \R^n$ by $(q, 0, \zeta, \xi_n)$ as
above, and put $\tau = \sqrt{|\zeta|^2 + \xi_n^2}$. We note that
$\xi_n$ is determined by $(q, \zeta, \tau)$ by $\xi_n =
\sqrt{\tau^2 - |\zeta|^2}$, since it is inward pointing.  The
coordinates $q, \zeta$ are symplectic, so the symplectic form on
$T^* \partial \Omega$ is $d \sigma = dq \wedge d \zeta. $ We then
relate the graph of the  billiard flow (\ref{PHI}) with initial
and terminal point  on the boundary to the billiard map (after $j$
reflections) by the formula
\begin{equation} \label{PHIBETA} \Phi^{T_j}(q, 0, \zeta, \xi_n) = (\tau \beta^j(q, \frac{\zeta}{\tau}), \xi_n'(q, \zeta, \xi_n)), \end{equation}
where $\xi_n' = \tau \sqrt{1 - |\beta^j(q, \frac{\zeta}{\tau})|^2}.
$  We also put \begin{equation} \label{gamma} \gamma(q, \zeta,
\tau) = \sqrt{1 - \frac{|\zeta|^2}{\tau^2}}, \;\; \text{and}\;\;  \gamma_1(q, \zeta) = \sqrt{1 - |\zeta|^2}.
\end{equation} It is the homogeneous (of degree zero) analogue of the function denoted by
$\gamma$ in \cite{HZ}.

Further,  we parameterize the canonical relation
$\Gamma^j_{\partial,+}$ of Lemma \ref{partial} using the billiard map
$\beta$ and its powers. We define the $j$th return time $T^j(q,
\xi)$ of the billiard trajectory in a codirection $(q, \xi) \in
T^*_q \Omega$ to be the length the $j$-link billiard trajectory
starting at $(q, \xi)$ and ending at a point $\Phi^{T^j(q, \xi)}
(q, \xi) \in T^*_{\partial \Omega} \Omega $. It is the same as $T^j_+(q, \xi)$.  Then we define
%\begin{equation} \label{PARAM} \iota_{\partial, j,+}:
%T^{* +}_{\partial \Omega} \R^n  \to T^*(\R \times \partial \Omega
%\times
%\partial \Omega), \;\;\;\iota_{ \partial, j}(q,  \xi) = (T^j(q, \xi),
%|\xi|, \Phi^{T^j(q, \xi)}(q,
%\xi)|_{T \partial \Omega}, (q, \xi |_{T_q \partial \Omega}) ),
%\end{equation}
%where $T^{* +}_{\partial \Omega} \R^n$ is the space of inward
%pointing covectors to $\Omega$.   It is somewhat more convenient
%to parameterize this canonical relation by
\begin{equation} \label{PARAMb} \iota_{\partial, j,+}:
 \R_+ \times T^*
\partial \Omega  \to T^*(\R \times \partial \Omega
\times
\partial \Omega), \;\;\;\iota_{ \partial, j}(\tau, q, \zeta) = (T^j(q, \xi(q,\zeta, \tau)),
\tau, ( \tau \beta^j (q, \frac{\zeta}{\tau})), q, \zeta),
\end{equation} where
$$\xi(q, \zeta, \tau)= \zeta + \xi_n \nu_q, \;\; \;\; |\zeta|^2 + |\xi_n |^2 = \tau^2. $$ The map (\ref{PARAMb}) parameterizes
$\Gamma^j_{\partial, +}$ of Lemma \ref{partial}.

\begin{prop} \label{PROP} In the coordinates $(\tau, q, \zeta) \in \R_+ \times T^* \partial \Omega$ of \eqref{PARAMb},
 the principal symbol of $$\chi_{T}(D_t,q', D_{q'}) \chi_{T}(D_t, q,
D_{q}) S^b_B(t, q', q)$$ on $\Gamma^j_{\partial, +}$
is as follows:
\begin{itemize}
\item in the Dirichlet case:
$$\sigma_{j,+}(q, \zeta, \tau) = C_{j,+}^D \chi_T(q, \frac{\zeta}{\tau}) \chi_T(\beta^j (q, \frac{\zeta}{\tau})) \gamma^{\half}(q, \zeta, \tau)
\gamma^{\half}(\tau \beta^j (q,\frac{\zeta}{\tau}), \tau))\tau |dq \wedge d
\zeta \wedge d \tau|^{\half},
$$
\item in the Neumann case:
$$\sigma_{j,+}(q, \zeta, \tau) = C_{j,+}^N \chi_T(q, \frac{\zeta}{\tau}) \chi_T(\beta^j (q, \frac{\zeta}{\tau})) \gamma^{-\half}(q, \zeta, \tau)
\gamma^{-\half}(\tau \beta^j (q,\frac{\zeta}{\tau}), \tau))\big(\langle \zeta, \beta^j(q,\frac{\zeta}{\tau})\rangle -\tau \big) |dq \wedge d
\zeta \wedge d \tau|^{\half},
$$where $C_{j,+}^B$ are certain constants (Maslov factors).
\end{itemize}

\end{prop}

\begin{proof}

We only show the computations in the Dirichlet case. The Neumann
case is very similar and uses (\ref{b}) which will produce an
additional factor of $\tau \langle\zeta,
\beta^j(q,\frac{\zeta}{\tau})\rangle-\tau^2$.

By Lemma \ref{BOUNDARYSYMBOL}, the principal symbol of $S_B(t)$
consists of four pieces at the boundary, one for each mode
$A^j_{\pm}, \widehat{A_{\pm}^j}$. The symbol for the $-$ mode of
propagation is equal to that for the  $+$ mode of propagation
 under the time reversal map $\xi \to -
\xi$. Further by part 2 of Lemma \ref{BOUNDARYSYMBOL}, the symbol at the boundary (adjusted by taking
normal derivatives in the Dirichlet case) is invariant under the
reflection map $\xi \to \hat{\xi}$ at the boundary due to the
boundary conditions. Hence we only calculate the
$A_+^j$ component and use the invariance properties to
calculate the symbol on the other components.

We therefore assume that the symbol of $S_B$ is $\frac{1}{2\tau}$ times the graph
half-density $|dt \wedge d x \wedge d \xi|^{\half}$ on
$\Gamma^j_+$.
 We need to compose this graph half-density on the left  by the symbol $\xi_n |dq
\wedge d \zeta \wedge d \xi_n|^{\half}$ of $r^D=r\; N$, and on the
right by the symbol $\xi'_n |dq' \wedge d \zeta' \wedge d
\xi'_n|^{\half}$ of the adjoint $r^{D*}=N^* r^*$. Therefore we compute the restriction of the $\Gamma^j_+$  component onto $\Gamma^j_{\partial, +}$ and we remember to multiply the symbol by $\xi_n \xi_n'= \tau^2 \gamma(q, \zeta, \tau)
\gamma(\tau \beta^j (q,\frac{\zeta}{\tau}), \tau))$ and also by $\frac{1}{2\tau}$ at the end.

It is simplest to use symbol algebra and pullback formulae to calculate it (see
\cite{DG}). The  composition is equivalent to the pullback of the symbol under
the pullback
\begin{equation} \label{PULLBACK} \Gamma^j_{\partial} =
(i_{\partial \Omega} \times i_{\partial \Omega})^* \Gamma^j,
\end{equation} of the canonical relation of the $S_B$ by the canonical inclusion map $$i_{\partial \Omega}
\times i_{\partial \Omega}: \R \times \partial \Omega \times
\partial \Omega \to \R \times \R^n \times \Omega.$$ We recall that
a map $f: X \to Y$ is transversal to $W \subset T^* Y$ if $df^*
\eta \not= 0$ for any $\eta \in W$. If $f: X \to Y$ is smooth and
$\Gamma \subset T^* Y$ is Lagrangian, and if $f$ and $ \pi: T^* Y
\to Y$ are transverse then $f^* \Gamma$ is Lagrangian. Since
$$(i_{\partial \Omega} \times
i_{\partial \Omega})^* (t, \tau, \Phi^t(q, \xi), q, \xi) = (t,
\tau, \Phi^t(q, \xi)|_{T \partial
\Omega}, q, \xi|_{T \partial \Omega} ) $$ at a point over  $(i_{\partial \Omega} \times
i_{\partial \Omega}) (t, q', q)$, and since $\tau = |\xi| \not=
0$, it is clear that $i_{\partial \Omega} \times i_{\partial
\Omega}$ is transversal to $\pi$.

We now claim that on the pullback of $\Gamma^j$, using the
parametrization (\ref{PARAMb}),
\begin{equation}\label{QUO} (i_{\partial \Omega} \times i_{\partial \Omega})^* |dt \wedge d x \wedge d
\xi|^{\half} = \gamma^{- \half}(q, \zeta, \tau)
\gamma^{-\half}(\tau\beta^j(q, \frac{\zeta}{\tau}), \tau) |dq \wedge d \zeta \wedge d
\tau|^{\half},
\end{equation} where $\gamma$ is defined in (\ref{gamma}). To see this, we use the pullback diagram
$$\begin{array}{cccc}
\Gamma^j  &\overset{\pi}{\longleftarrow}& F & \overset{\alpha}{\xrightarrow{\hspace*{1.2cm}}}
(i_{\partial \Omega} \times i_{\partial \Omega})^* \Gamma^j
\subset T^*(\R \times \partial \Omega \times \partial \Omega)
\\ &&\\ i \downarrow &  & \pi \downarrow &
\\ && \\
T^*(\R \times \R^n \times \Omega)  & \overset{\pi}{\longleftarrow} & {\mathcal
N}^*(\mbox{graph}(i_{\partial \Omega} \times i_{\partial \Omega}
)) &
\end{array}$$
Here, $F$ is the fiber product, ${\mathcal
N}^*\mbox{graph}(i_{\partial \Omega} \times i_{\partial \Omega})$
is the co-normal bundle to the graph, and the map $\alpha: F \to
(i_{\partial \Omega} \times i_{\partial \Omega})^* \Gamma^j$ is the
natural projection to the composition (see \cite{DG}). Since the
composition is transversal, $D \alpha$ is an isomorphism (loc.
cit.). The graph of $i_{\partial \Omega} \times i_{\partial
\Omega}$ is the set $\{(t, q,  q',
 t, q,  q' ): (t, q, q') \in \R \times \partial \Omega \times \partial \Omega\}$ and its
conormal bundle is (in the Fermi normal coordinates),
$$\begin{array}{lll} {\mathcal N}^*(\mbox{graph}(i_{\partial \Omega} \times i_{\partial
\Omega})) &  = & \{(t, \tau , q, \zeta, q', \zeta',  t, - \tau, q, -
\zeta + \xi_n, q', - \zeta' + \xi_n'), \;\;\\ && \\ && (q, \zeta,
\xi_n), (q', \zeta', \xi_n') \in T^*_{\partial \Omega}\R^n\}
\\ && \\
& \subset & T^*(\R \times \partial \Omega \times
\partial \Omega \times \R \times \mathbb R^n \times \mathbb R^n),\end{array}
$$ The half density produced by the pullback diagram
takes the exterior tensor product of the canonical half density
$$|dt \wedge d \tau \wedge dq \wedge d \zeta \wedge d \xi_n \wedge
d \xi_n' \wedge dq' \wedge d \zeta' |^{\half}$$ on $ {\mathcal
N}^*(\mbox{graph}(i_{\partial \Omega} \times i_{\partial \Omega}))
$  and $$|dt' \wedge dx' \wedge d \xi'|^{\half}, \;\;\mbox{ on}\;
\Gamma^j \subset T^*(\R \times \mathbb R^n \times \Omega)$$
  at a point of the fiber product (where the
  $T^*(\R \times \mathbb R^n \times
\Omega)$ components are equal) and divides by the canonical half
density
$$|dt' \wedge d\tau' \wedge dq' \wedge d \zeta' \wedge d x_n' \wedge
d \xi_n' \wedge dx' \wedge d \xi' |^{\half}$$ on the common $T^*
\R \times T^* \R^n \times T^* \Omega$ component.

Since $\tau' = \tau$, the  factors of $|dt' \wedge d \tau' \wedge
dq' \wedge d \zeta' \wedge d \xi_n' \wedge dx' \wedge d \xi'
|^{\half}$ cancel in the quotient half-density, leaving the half
density
$$\frac{|dt  \wedge  d q \wedge d \zeta \wedge d \xi_n |^{\half}}{| d x_n'|^{\half}} $$
on the composite. The numerator is a half-density on   $\R \times
T^*_{\partial \Omega} \R^n$. We write it  more intrinsically in
the following Lemma. Note that it explains the first of our two
$\gamma$ factors.

\begin{lem} \label{QUOTCALC} Let $\Phi=\Phi_+$ be the parametrization (\ref{PHI}) and $\omega_{T^*\mathbb R^n}$ be the canonical symplectic form of $T^*\mathbb R^n$. Then $|dt \wedge d q \wedge d
  \zeta \wedge d \xi_n|^{\half} = |\frac{\xi_n}{\sqrt{|\zeta|^2  + \xi_n^2}}|^{-\half} \left|  \Phi^{*}
  \omega_{T^* \R^n} \right|^{\half} $ as half-densities on $\R \times T^*_{\partial \Omega} \R^n$. \end{lem}

  \begin{proof} We have,
  $$\begin{array}{lll} \frac{\Phi^{*} \omega_{T^* \R^n}}{dt \wedge d q \wedge d
  \zeta\wedge d \xi_n} & = & \omega_{T^* \R^n} (\frac{d}{dt} \Phi^t(q,
  \zeta, \xi_n), d \Phi^t \frac{\partial}{\partial q_j},  d \Phi^t \frac{\partial}{\partial
  \zeta_j},  d \Phi^t \frac{\partial}{\partial \xi_n} ) \\ && \\
  & = & \omega_{T^* \R^n} (H_g,  \frac{\partial}{\partial q_j},  \frac{\partial}{\partial
  \zeta_j},   \frac{\partial}{\partial \xi_n} )  \\ && \\
  & = &  \frac{\xi_n}{\sqrt{|\zeta|^2  + \xi_n^2}} \omega_{T^* \R^n} (\frac{\partial}{\partial x_n},  \frac{\partial}{\partial q_j},  \frac{\partial}{\partial
  \zeta_j},   \frac{\partial}{\partial \xi_n} ) \\ && \\
  & = & \frac{\xi_n}{\sqrt{|\zeta|^2  + \xi_n^2}}.  \end{array} $$
since $\frac{d}{dt} \Phi^t(q,
  \eta, \xi_n) = H_g =  \frac{\xi_n}{\sqrt{|\zeta|^2  + \xi_n^2}} \frac{\partial}{\partial x_n} + \cdots $ is the Hamilton vector
  field of $g  = \sqrt{g^2}, g^2 = \xi_n^2 + (g')^2$ where $\cdots$  represent vector fields in
  the span of $\frac{\partial}{\partial q_j},  \frac{\partial}{\partial
  \zeta_j},   \frac{\partial}{\partial \xi_n}$. Finally, we use that   $d \Phi^t$ is symplectic
  linear and that $q, x_n, \zeta, \xi_n$ are symplectic
  coordinates. Note that we are evaluated the symplectic volume
  form at the domain point, not the image point.

  \end{proof}

Next we take consider the points in the image of $\Phi$ on $\R
\times T^*_{\partial \Omega} \R^n$ where $x_n' = 0$ and take the
quotient by $|d x_n'|^{\half}$,  resulting in a half density on
$\Gamma^j_{\partial}$. The next Lemma explains the origin of the
second $\gamma$ factor.

  \begin{lem}\label{BIGFACTOR} In the subset  $\Gamma^j_{\partial} \subset \Phi(\R
\times T^*_{\partial \Omega} \R^n)$ where $x_n' = 0$ and where $t
= T^j$,  we have (in the parameterizing coordinates
(\ref{PARAMb})),
$$\frac{|dt \wedge d q \wedge d \zeta \wedge d \xi_n |^{\half}}{|
d x_n'|^{\half}} = \left| ((\beta^j)^* \gamma^{-1})  d q \wedge d
\eta \wedge d \tau \right|^{\half}. $$
  \end{lem}

  \begin{proof}

 By the previous Lemma, it suffices to rewrite $$  |d x_n'|^{-\half} \left|  \Phi^*
  \omega_{T^* \R^n} \right|^{\half}$$
  in the coordinates $(\tau, q, \eta)$  of
$\iota_{\partial, j}$ in (\ref{PARAMb}).  We observe
that $x_n' = \Phi^*
  x_n$. Hence
 $$\begin{array}{lll}  |d x_n'|^{-\half} \left| \Phi^{*}
  \omega_{T^* \R^n} \right|^{\half} & = &  \left| \Phi^{*}
  \frac{\omega_{T^* \R^n}}{|d x_n|} \right|^{\half} \\ && \\
  &  = &   \left| ((\beta^j)^* \gamma^{-1}) d q \wedge d \zeta \wedge d \tau\right|^{\half}. \end{array} $$
  In the last line, we use (\ref{PHIBETA}),  that
  $$\frac{\omega_{T^* \R^n}}{|dx_n|} = |dq \wedge d \zeta \wedge d
  \xi_n|, $$
    and that $\beta$ is symplectic. Indeed, by (\ref{PHIBETA}),
    $$\begin{array}{l} \Phi^{*} (dq \wedge d \zeta \wedge d \xi_n) = \tau (\beta^j)^* (dq
    \wedge d \frac{\zeta}{\tau}) \wedge  \Phi^{*} d \xi_n \\ \\ = \tau (\beta^j)^* (dq
    \wedge d \frac{\zeta}{\tau}) \wedge  \Phi^{*} d \sqrt{\tau^2 -
    |\zeta|^2} \\  \\
    = dq \wedge d \zeta \wedge \Phi^{*} \frac{ \tau d \tau}{\sqrt{\tau^2 -
    |\zeta|^2}} = ((\beta^j)^* \gamma^{-1}) dq \wedge d \zeta \wedge d \tau.  \end{array}$$
    Note that $\tau (\beta^j)^* (dq
    \wedge d \frac{\zeta}{\tau}) = dq \wedge d \zeta |_{\beta^j(q, \zeta)}$.

  \end{proof}

Combining Lemma \ref{BIGFACTOR} with Lemma \ref{QUOTCALC}
completes the proof of (\ref{QUO}) and Proposition \ref{PROP}.

\end{proof}

%\begin{rem} The  symbol calculation is closely related to Proposition
%6.1 of \cite{HZ}. It  is the homogeneous analogue, where one
%transforms the Neumann/Dirichlet Green's functions into
%Neumann/Dirichlet wave kernels by a semi-classical Fourier
%transform $\lambda \to t$. But  the computation in \cite{HZ} only
%involved layer potentials and not the wave group. The
%Dirichlet/Neumann wave kernels may be expressed (modulo smoother
%kernels) in terms of layer potentials and the boundary integral
%operators they induce, so we expect the same type of formula for
%the symbol of the wave kernels.
%
%\end{rem}

\subsection{\label{SYMPTRACE} Trace along the  boundary:  composition with $\pi_* \Delta^*$}

We now take the trace along the boundary of this operator.
Analogously to  \cite{DG,GM,MM}, we define $\Delta: \R \times
\partial \Omega \to \R \times \partial \Omega \times \partial
\Omega$ to be the diagonal embedding and $\pi_*$ to be integration over $\partial \Omega$.

\begin{lem}\label{symboloftrace} If the fixed point sets of period $T$ of $\beta^k$ are clean for all
$k$ and form a submanifold $F_T$ of $B^*\partial \Omega$ of dimension $d$ (with connected components $\Gamma$), then
$$\pi_* \Delta^* \dot{\rho} \, \chi_{T}(D_t, q', D_{q'}) \chi_{T}(D_t, q, D_{q})S^b_B(t, q', q) \in I^{\frac{d}{2} + \half + 1 -
\frac{1}{4}} (\R, \; T^*_T \mathbb R),
$$ where $$T^*_T \mathbb R= \cup_{\pm}\Lambda_{T, \pm}=\cup_{\pm}\{(T, \pm \tau):\, \tau \in \mathbb R_{+}\}, $$
and its principal symbol on $\Lambda_{T,\pm}$ is given by
$$
c^\pm \tau^{\frac{d + 2}{2}} \;
\sqrt{d\tau}, $$ where $$c^\pm= \sum_{\Gamma \subset F_T} C^{\pm}_{\Gamma} \int_{\Gamma} \dot{\rho}\; \gamma_1 \;  d \mu_{\Gamma}$$ and $c^-=\bar{c^+}$ the complex conjugate of $ c^+$.
\end{lem}

\begin{proof} The calculation of the principal symbol of the trace
of a Fourier integral operator in \cite{DG} is valid for the
boundary restriction of the wave kernel, since it only uses that
it is $\pi_* \Delta^*$ composed with a Fourier integral kernel
with a known symbol and canonical relation.  Hence we follow the
proof closely and refer there for further details.

As in \cite{GM}, the composition of $\pi_* \Delta^*$ with
\begin{equation} \label{THING} \dot{\rho} \chi_{T}(D_t, q', D_{q'}) \chi_{T}(D_t, q, D_{q})S^b_B(t, q', q) \end{equation} is clean if and only if
the fixed point set of $\beta^k$ corresponding to periodic orbits
of period $T$ is clean. When the fixed point set has dimension $d$
in the ball bundle $B^*\partial \Omega$, composition with $\pi_* \Delta^* $
adds $\frac{d}{2}$ to the order (see \cite{DG}, (6.6)). Combining
with Lemma \ref{partial}, we obtain the order $\frac{d}{2} + \half
+ 1 - \frac{1}{4}$.

Hence under the cleanliness assumption, it follows that  $\delta \;Tr \;
\cos t \sqrt{-\Delta_B}$ is a Lagrangian distribution on $\R$ with
singularities at $t \in Lsp(\Omega)$. As discussed in \cite{DG} (loc.
cit.) for the upper/lower half lines $\Lambda_{T,\pm}$ in $T^*_T \R$,
$I^{\frac{d}{2}+ \frac{5}{4}}(\R, \Lambda_{T,\pm})$ consists of multiples
of the distribution
$$\int_0^{\infty} \tau^{\frac{d+2}{2}} e^{\pm i \tau (t - T)} d\tau = (t - T \pm i 0)^{- \frac{d + 4}{2}}.$$
The principal symbol of this Fourier integral distribution is
$\tau^{\frac{d + 2}{2}} \; \sqrt{d\tau}$. Therefore to conclude the
Lemma we only need to compute the coefficients of this symbol in
the trace.

This coefficient is computed in a universal way from the
principal symbol of (\ref{THING}) computed from Proposition \ref{PROP}.
Following the proof of \cite{DG}, the coefficient of $\tau^{\frac{d + 2}{2}} \; \sqrt{d\tau}$ is
$$c^{\pm}= \sum_{\Gamma \subset F_T} C^\pm_\Gamma \int_{\Gamma} \dot{\rho}\; \gamma_1 \;  d \mu_{\Gamma}, $$
where $F_T$ is the fixed point set of $\beta$ (and its powers) in $B^*\partial \Omega$. The sum is over the connected components $\Gamma$
of $F_T$
and $d\mu_\Gamma$ is the canonical density on the fixed point component $\Gamma$
defined in Lemma 4.2 of \cite{DG}. We note that the distribution $c^+ (t - T + i 0)^{- \frac{d + 4}{2}} +c^-(t - T - i 0)^{- \frac{d + 4}{2}} $ is real only if $c^-=\bar{c^+}$. This completes the proof of the Lemma.

\end{proof}

The Lemma also completes the proof of the Theorem. We close the
section with a remark:

%\begin{rem} If $P_{\Gamma_T} =
%d \beta: T_{\zeta} T^* \partial \Omega \to T_{\zeta} T^* \partial
%\Omega$ is the derivative of the billiard map at a point $\zeta
%\in \Gamma_T$, then $\ker (I - P_{\Gamma_T})$ has a canonical
%density or the form $| \det(I - P_{\Gamma_T}^{\#})|^{-\half}
%|\Omega|^r$, where $r = \half \dim \ker (I - P_{\Gamma_T})$.
%\end{rem}

\begin{rem} As a check on the order, we note that for the wave
trace in the interior and for non-degenerate closed trajectories,
the singularities are of order $(t - T + i 0)^{-1}$. When the
periodic orbits are degenerate and the unit vectors in the fixed
point sets have dimension $d$, the singularity increases  to order
$(t - T + i 0)^{-1 - \frac{d}{2}}$. If we formally take the
variation of the wave trace, the singularity should increase to
order $(t - T + i 0)^{-1 - \frac{d}{2} - 1}$.

In comparison,  the boundary trace in the Dirichlet case involves
two extra derivatives of the wave kernel and composition with
$(-\Delta)^{- \half}$. Compared to the interior trace, this adds one
net derivative and  order to the trace singularity. We claim that
the restriction to the boundary  does not further change the order
compared to the interior trace. This can be seen by considering
the method of stationary phase for oscillatory integrals with
Bott-Morse phase functions, whose non-degenerate critical
manifolds are transverse to the boundary. If we restrict the
integral to the boundary, we do not change the number of phase
variables in the integral, but we simultaneously decrease the
number of variables by one and the dimension of the fixed point
set by one. The number of non-degenerate directions stays the
same. It  follows that  the singularity order of the variational
trace  goes up by one overall unit compared to the interior trace,
consistently with the formal variational calculation.

\end{rem}

%
%\begin{rem} The difference in the Dirichlet vs. Neumann cases is due to the
%different derivatives taken of $(\Delta^{-\half} U)^b(t, q, q')$
%in the Neumann case. In the Dirichlet case we multiply the
%symplectic area measure $d \sigma$ by $\gamma^{-1}$ and then by
%$\gamma^2$ due to the normal derivatives. In the Neumann case, the
%Hadamard formula gives two terms of the same order, one of which
%involves no derivatives and gives $\gamma^{-1} d \sigma$ and one
%which multiples by the symbol of the square
%$|\frac{\eta}{\tau}|^2$ of the tangential derivative. Hence the
%%measure in the Neumann case is $(1 - |\eta|^2) \gamma^{-1}
%d\sigma$, the same as in the Dirichlet case.
%\end{rem}
%

 \section{Case of the ellipse and the proof of Theorem \ref{RIGID}}

In this section we let $\Omega_0$ be an ellipse. In this case, the fixed point sets are clean fixed
point sets for $\Phi^t$ in $T^*\Omega_0$, resp. for $\beta$ in
$B^*\partial \Omega_0$ (See  \cite{GM} Proposition 4.3). In fact the fixed point sets $F_T$ of $\beta$ in $B^*\partial \Omega_0$ form a one dimensional manifold. Thus $d=1$ and corollary \ref{VARWTell} follows.

% Thus, $d = 1$ for the trace on the
%boundary in the ellipse case.

%\begin{lem} \label{SYMBOL} Let $\Omega_0$ be an ellipse,
%and  $T \in Lsp(\Omega_0)$ with $T$ not a multiple of $|\partial
%\Omega|$.  Then the principal symbol of $Tr \dot{\rho} S^b(t)$ at
%$t = T$ is given by $\int_{\Gamma_T} \dot{\rho}\; \gamma\;
%d\mu_{\Gamma_T}$ where $d\mu_{T}$ is the
%Leray measure on $\Gamma_T$.
%\end{lem}

%As above, we denote by $\Phi^t: S^* \Omega \to S^*\Omega$ the
%generalized geodesic flow (or broken billiard flow) of the ellipse
%$\Omega_0$, and we denote by $\beta: B^* \partial \Omega_0 \to B^*
%\partial \Omega_0$ the associated billiard map. The broken  geodesic flow extends by homogeneity (degree one) to $T^*
%\Omega - 0$.  We denote the Hamiltonian vector field of the
%Euclidean norm function $g$ by $H_g$. The symbol of the trace
%involves singularities due to the fact that the map $\Phi: \R
%\times T^*_{\partial \Omega} \Omega - 0 \to T^* \Omega - 0$
%defined by $(t, x, \xi) \to \Phi^t(x, \xi)$ is a homogeneous
%folding map with folds along $\R \times T^*
%\partial \Omega$.
As is well-known, both the billiard flow and
billiard map of the ellipse are completely integrable. In
particular,  except for certain exceptional trajectories,   the
periodic points of period $T$ form a Lagrangian tori in $S^*
\Omega_0$, and the homogeneous extensions of the Lagrangian tori are
cones in  $T^* \Omega_0$. The exceptions are the two bouncing ball
orbits through the major/minor axes and the trajectories which
intersect the foci or glide along the boundary. The fixed point
sets of $\Phi^T$ intersect the co-ball bundle $B^*
\partial \Omega_0$ of the boundary in the fixed point sets of the
billiard map $\beta: B^* \partial \Omega_0 \to B^* \partial \Omega_0$
(for background we refer to \cite{PS,GM,GM2,HZ,TZ2} for instance).
Except for the exceptional orbits, the fixed point sets are real
analytic curves. For the bouncing ball rays, the associated fixed
point sets are non-degenerate fixed points of $\beta$.

Since the final step of the proof uses results of \cite{GM}, we
briefly review the description of the billiard map of the ellipse
$\Omega_0:=\frac{x^2}{a} + \frac{y^2}{b} = 1$ (with $a > b > 0$)  in
that article. In the interior, there exist for each $0 < Z \leq b$
a caustic set given by a confocal ellipse
$$\frac{x^2}{\epsilon + Z} + \frac{y^2}{Z} = 1$$
where $\epsilon = a - b$,  or for $- \epsilon < Z < 0$ by a
confocal hyperbola. Let $(q, \zeta)$ be in $B^*\partial \Omega_0$ and let $(q,\xi)$ in $S^*\Omega_0$ be the unique inward unit normal to boundary that projects to $(q,\zeta)$. The line segment $(q,r\xi)$ will be tangent to a unique confocal ellipse or hyperbola (unless it intersects the foci). We then define the function $Z(q, \zeta)$ on $B^* \partial \Omega_0$ to be the corresponding $Z$.  Then $Z$ is a $\beta$-invariant function
and its level sets $\{Z = c\}$ are the invariant curves of
$\beta$. The invariant Leray form on the level set is denoted
$du_Z$ (see \cite{GM}, (2.17), i.e. the symplectic form of $B^*\partial \Omega_0$
is $dq \wedge d \zeta = dZ \wedge du_Z$. A level set has a rotation number and the
periodic points live in the level sets with rational rotation
number. As it is explained in \cite{GM} (page 143) the Leray form $du_Z$ restricted to a connected component $\Gamma$ of $F_T$ is a constant multiple of the canonical density $d\mu_\Gamma$.

% The fixed point set of a given period $T$ is a
%certain level set $\{I = \alpha_T\}$ of the $\Z_2 \times
%\Z_2$-invariant Hamiltonian on $B^* \partial \Omega_0$,
% $$I :=  p_{\vartheta}^2 + c^2 \cos^2
%\vartheta.$$ The level sets $\{I = \alpha\}$ are $\beta$-invariant
%curves and up to the $\Z_2 \times \Z_2$ symmetry they are
%irreducible invariant curves, i.e. are not unions of invariant
%components. There is a natural invariant measure $d \mu_{\alpha}$
%on each component of $\{I = \alpha\}$, namely the Leray quotient
%measure
 %$ d\mu_{\alpha}= \frac{d \vartheta \wedge d p_{\vartheta}}{d
%I}$ of the symplectic area form by $d I$. They are invariant under
%the Hamilton flow of $I$ and under the billiard map $\beta$.

 As mentioned in the introduction, the  well-known obstruction to using trace formula calculations
 such as in Proposition \ref{VARWTintro} is multiplicity in the length
 spectrum, i.e. existence of several connected components of
 $F_T$. A higher dimensional component is not itself a
 problem, but there could exist cancellations among terms coming
 from components with different Morse indices, since the
 coefficients $C_{\Gamma}$ are complex. This problem arose earlier
 in the spectral theory of the ellipse in \cite{GM}. Their key
 Proposition 4.3 shows that there are is a sufficiently large set
 of lengths $T$ for which $F_T$ has one component up to $(q, \zeta) \to (q, -\zeta)$ symmetry.
 Since it is crucial here as well, we state the relevant part:

 \begin{prop}\label{GM}  (see \cite{GM}, Proposition 4.3): Let $T_0 =
 |\partial \Omega_0|$. Then for every interval $(m T_0 - \epsilon,
 m T_0)$ for $m = 1, 2, 3, \dots$ there exist infinitely many
 periods $T \in Lsp(\Omega_0)$ for which $F_T$ is the union
 of two invariant curves which are mapped to each other by $(q,\zeta)
 \to (q, -\zeta)$. \end{prop}
Since for an isospectral deformation $\delta\; Tr \cos( t\sqrt{-\Delta})=0$, we obtain from Proposition \ref{VARWTintro} the following

\begin{cor} \label{MAIN} Suppose we have an isospectral deformation of an ellipse $\Omega_0$ with velocity $\dot{\rho}$. Then for each $T$ in Proposition \ref{GM} for which $F_T$ is the union
 of two invariant curves $\Gamma_1$ and $\Gamma_2$ which are mapped to each other by $(q, \zeta) \to (q, -\zeta)$ we have
$$\int_{\Gamma_j} \dot{\rho} \; \gamma_1 \; du_Z = 0, \qquad \quad j=1,2.$$
 \end{cor}
\begin{proof}
From Proposition \ref{VARWTintro} we get
$$\Re \big\{ \big( \sum_{j=1}^2 C_{\Gamma_j} \int_{\Gamma_j} \dot{\rho}\; \gamma_1 \;  d \mu_{\Gamma_j}
\big) \; (t - T+ i 0)^{- 2 - \frac{d}{2}} \big\}=0.$$ Since $\dot{\rho}$ and $\gamma_1$ are invariant under the time reversal map $(q,\zeta) \to (q,-\zeta)$, the two integrals are identical. Also by directly looking at the stationary phase calculations it can be shown that the Maslov coefficients $C_{\Gamma_1}$ and $C_{\Gamma_2}$ are also the same. Thus the corollary follows.
\end{proof}

\subsection{Abel transform}

The remainder of the proof of Theorem \ref{RIGID} is identical to
that of Theorem 4.5 of \cite{GM} (see also \cite{PT}). For the sake of completeness, we
sketch the proof.

\begin{prop}\label{RHODOT}  The only $\Z_2 \times \Z_2$ invariant function
$\dot{\rho}$ satisfying the equations of Corollary \ref{MAIN} is
$\dot{\rho} = 0$. \end{prop}

\begin{proof} First, we may assume $\dot{\rho} = 0$ at the
endpoints of the major/minor axes, since the deformation preserves
the $\Z_2 \times \Z_2$ symmetry and we may assume that the
deformed bouncing ball orbits will not move and are aligned with the original ones.
Thus $\dot{\rho}(\pm \sqrt{a}) = \dot{\rho}(\pm \sqrt{b}) = 0$.

The Leray measure may be explicitly evaluated (see $2.18$ in \cite{GM}). By a change of
variables with Jacobian $J$, and using the symmetric properties of $\dot{\rho}$, the integrals become
\begin{equation} \label{F} A(Z) = \int_a^b \frac{\dot{\rho}(t)\; \gamma_1 \; J(t) dt}{\sqrt{t -
(b - Z)}}. \end{equation} for an infinite sequence of $Z$
accumulating at $b$. Since $0 < a < b $, the  function $A(Z)$ is
smooth in $Z$ for $Z$ near $b$. It vanishes infinitely often in
each interval $(b - \epsilon, b)$, hence is flat at $b$. The $k$th
Taylor coefficient at $b$ is \begin{equation} \label{FT}
A^{(k)}(b) = \int_a^b \dot{\rho}(t)\; \gamma_1 \; J(t) t^{- k -
\half} dt = 0.
\end{equation}
Since the functions $t^{-k}$ span a dense subset of $C[a, b]$, it
follows that $\dot{\rho} \equiv 0.$

\end{proof}

\subsection{\label{FLAT} Infinitesimal rigidity and flatness}

We now show that infinitesimal rigidity
implies flatness and prove Corollary
\ref{RIGIDCOR}. As mentioned, the Hadamard variational formula is valid for any $C^1$ parametrization  $\Omega_{\alpha(\epsilon)}$
of the domains $\Omega_\epsilon$. For each one we have
$\delta \rho_{\alpha(\epsilon)} (x) \equiv 0$.

Assume $\rho_\epsilon(x)$ is not flat at $\epsilon=0$ and let $\epsilon^k$ be the first non-vanishing term in the Taylor expansion of $\rho_\epsilon(x)$ at $\epsilon=0$. Then
\begin{equation} \label{TAYLORk} \rho_\epsilon(x) = \epsilon^k \frac{\rho^{(k)}(x)}{k!} + \epsilon^{k + 1} \frac{\rho^{(k + 1)}(x)}{(k+1)!} +
\cdots. \end{equation} We then reparameterize the family by $\epsilon \to
\alpha(\epsilon): = \epsilon^{\frac{1}{k}}$ so that
$$\rho_{\alpha(\epsilon)}(x) = \frac{\rho^{(k)}(x)}{k!} \epsilon + O(\epsilon^{1 + 1/k}). $$
By Hadamard's variational formulae we get
$\delta \rho_{\alpha(\epsilon)} (x)= \rho^{(k)}(x) \equiv 0$, a contradiction.


\begin{thebibliography}{HHHH}

\bibitem[A]{A} E. Y. Amiran,
A dynamical approach to symplectic and spectral invariants for
billiards.  Comm. Math. Phys. 154 (1993), no. 1, 99--110.

\bibitem[A2]{A2} E. Y. Amiran,
Noncoincidence of geodesic lengths and hearing elliptic quantum
billiards. J. Statist. Phys. 85 (1996), no. 3-4, 455--470.


\bibitem[BB]{BB} V. M. Babich and V. S. Buldyrev, {\it Short
wavelength Diffraction Theory}, Springer Series Wave Phenomena 4
(1991), Springer Verlag.

%\bibitem[CR]{CR} P. S. Casas and R. Ramirez-Ros, The frequency
% map for billiards inside ellipsoids (arXiv:1004.5499, 2010).



\bibitem[Ch]{Ch} J. Chazarain,
Construction de la param\'etrix du probl\`eme mixte hyperbolique
pour l'\'equation des ondes.  C. R. Acad. Sci. Paris Sér. A-B 276
(1973), A1213--A1215.

\bibitem[Ch2]{Ch2} J. Chazarain, Param\'etrix du probl\`eme mixte
pour l'\' equation des ondes \`a l'int\'erieur d'un domaine
convexe pour les bicaract\'eristiques, Journ\'ees \` Equations aux
d\'eriv\'ees partielles (1975), p. 165--181.



\bibitem[CdV]{CdV} Y. Colin de Verdi\`ere, Sur les longueurs des trajectoires p\'eriodiques d'un billard, in: Dazord, Desolneux (eds.): G\'eom\'etrie symplectique et de contact, Sem. Sud-Rhod. Géom. (1984), 122 139.

 %\bibitem[CCS]{CCS} B. Crespi, S-J. Chang and K.-J Shi, Elliptical
% billiards and hyperelliptic functions, J. Math. Phys. 34 (1993)
% 2257-2289.

\bibitem[DH]{DH} J. J. Duistermaat and L. H\"ormander,
Fourier integral operators II. Acta Math. 128 (1972), no. 3-4, 183--269.

\bibitem[DG]{DG} J. J. Duistermaat and V. W.  Guillemin,
The spectrum of positive elliptic operators and periodic
bicharacteristics. Invent.\ Math. 29 (1975), no.\ 1, 39--79.

 \bibitem[EZ]{EZ} L.C. Evans and M. Zworski, {\it Lectures on
 Semi-Classical Analysis}, online at
 http://math.berkeley.edu/~zworski.

%\bibitem[K]{K} R. Kolodziej, The rotation number of some
%transformations related to billiards in an ellipse, Studia Math 81
%(1985), 293- 302.


%\bibitem[MM]{MM} S. Marvizi and R. B.  Melrose,  Spectral invariants of convex planar regions. J. Differential Geom. 17 (1982), no. 3, 475--502.


\bibitem[FO]{FO}
D. Fujiwara and S. Ozawa,  The Hadamard variational formula for
the Green functions of some normal elliptic boundary value
problems. Proc. Japan Acad. Ser. A Math. Sci. 54 (1978), no. 8,
215--220.

\bibitem[FTY]{FTY} D. Fujiwara, M.  Tanikawa, and S. Yukita,  The
spectrum of the Laplacian and boundary perturbation. I. Proc.
Japan Acad. Ser. A Math. Sci. 54 (1978), no. 4, 87--91.

\bibitem[G]{G} P. R. Garabedian, {\it Partial differential equations}. AMS Chelsea Publishing, Providence, RI, 1998.

%\bibitem[Gh]{Gh} M.  Ghomi, Shortest periodic billiard trajectories in convex bodies. Geom. Funct. Anal. 14 (2004), no. 2,
%295--302.

\bibitem[Gu]{Gu} V. Guillemin,  Wave-trace invariants. Duke Math. J. 83
(1996), no. 2, 287--352.

\bibitem[GK]{GK} V. Guillemin and D. Kazhdan,  Some inverse spectral results for negatively curved $2$-manifolds. Topology 19 (1980), no. 3, 301--312.

\bibitem[GM]{GM} V. Guillemin and R.B.  Melrose. An inverse spectral result for elliptical regions in $\Bbb{R}^2$. Adv. Math. 32 (1979), 128--148.

\bibitem[GM2]{GM2} V. Guillemin and R. B. Melrose, The Poisson summation formula for manifolds with boundary. Adv. in Math. 32 (1979), no. 3,
204--232.

\bibitem[HZ]{HZ} A. Hassell and S.  Zelditch, Quantum ergodicity of boundary values of eigenfunctions. Comm. Math. Phys. 248 (2004), no. 1, 119--168.

\bibitem[HeZ]{HeZ} H. Hezari and S. Zelditch, Inverse spectral problem for
analytic ${\bf Z}_2^n$-symmetric domains in ${\bf R}^n$, Geom.
Funct. Anal. 20 (2010), 160-191 (arXiv:0902.1373).

\bibitem[Ho]{Ho}  L. H\"ormander, {\em The Analysis of Linear Partial
Differential Operators\/}, Volumes III-IV, Springer-Verlag Berlin
Heidelberg, 1983.

\bibitem[ISZ]{ISZ} A. Iantchenko, J.  Sj\"ostrand, and M.  Zworski,
Birkhoff normal forms in semi-classical inverse problems. Math.
Res. Lett. 9 (2002), no. 2-3, 337--362.

%\bibitem[K]{K} Y. Kannai,  Off diagonal short time asymptotics for fundamental solutions of diffusion equations.
% Commun. Partial Differ. Equations 2 (1977), no. 8, 781–830.

%\bibitem[KR]{KR} J. B. Keller and S. I. Rubinow, Asymptotic Solution of Eienvalue
%Prolbems, Annals of Physics 9 (1960), 24-73.

\bibitem[MM]{MM} S. Marvizi and R. B.  Melrose, Spectral invariants of convex planar regions. J. Differential Geom. 17 (1982), no. 3,
475--502.


\bibitem[M]{M} R. B. Melrose, Isospectral sets of drumheads are
compact in $C^{\infty}$, unpublished  MSRI preprint.

\bibitem[O]{O} S. Ozawa, Hadamard's variation of the Green kernels of heat equations and their traces. I. J. Math. Soc. Japan 34 (1982), no. 3, 455--473.

 \bibitem[Pee]{Pee} J. Peetre,
On Hadamard's variational formula. J. Differential Equations 36
(1980), no. 3, 335--346.

\bibitem[PS]{PS} V. M. Petkov and L.  Stoyanov,
{\it Geometry of reflecting rays and inverse spectral problems.}
Pure and Applied Mathematics (New York). John Wiley $\&$ Sons,
Ltd., Chichester, 1992.

 \bibitem[PT]{PT} G.  Popov and P. Topalov, Liouville billiard tables and an inverse spectral result.
 Ergodic Theory Dynam. Systems 23 (2003), no. 1, 225--248.

  \bibitem[PT2]{PT2} G. Popov and P. Topalov,
 Invariants of isospectral deformations and spectral
rigidity (arXiv:0906.0449).




\bibitem[Sib]{Sib}  K. F. Siburg, Aubry-Mather theory and the inverse spectral problem for planar convex domains. Israel J. Math. 113 (1999),
285--304.

\bibitem[Sib2]{Sib2} K. F. Siburg,  Friedrich Symplectic invariants of elliptic fixed points. Comment. Math. Helv. 75 (2000), no. 4, 681--700.

\bibitem[Sib3]{Sib3} K. F. Siburg, {\it  The principle of least action in geometry and dynamics}. Lecture Notes in Mathematics, 1844. Springer-Verlag, Berlin,
2004.

%\bibitem[S]{S} M. Sieber,
%Semiclassical transition from an elliptical to an oval billiard.
% J. Phys. A 30 (1997), no. 13,
%4563--4596.
%
%\bibitem[Tab]{Tab} M. B. Tabanov,  Separatrices splitting for Birkhoff's billiard in symmetric convex domain, closed to an ellipse. Chaos 4 (1994), no. 4, 595--606.
%

%\bibitem[Ta]{Ta}  S. Tabachnikov,  Billiards. Panor. Synth. No. 1 (1995),

%\bibitem[T]{T} M. B. Tabanov,
%Separatrices splitting for Birkhoff's billiard in symmetric convex
%domain, closed to an ellipse. (English summary) Chaos 4 (1994),
%no. 4, 595--606.

\bibitem[T]{tataru} D. Tataru, {\em On the regularity of boundary traces for the wave equation},
Ann. Scuola Norm. Sup. Pisa Cl. Sci.   {\bf 26}  (1998), 185--206.



%\bibitem[TZ]{TZ} J. A. Toth and S.  Zelditch,  Riemannian manifolds with uniformly bounded eigenfunctions. Duke Math. J. 111 (2002), no. 1,
%97--132. $L^p$ norms of eigenfunctions in the completely
%integrable case. Ann. Henri Poincar\'e 4 (2003), no. 2, 343--368.

\bibitem[TZ2]{TZ2} J. A. Toth and S. Zelditch,  Quantum ergodic restriction theorems, I: interior
hypersurfaces in domains with ergodic billiards, arXiv:1005.1636.

\bibitem[TZ3]{TZ3} J. A. Toth and S. Zelditch, Quantum ergodic restriction theorems, II: manifolds without
boundary (arXiv:1104.4531).

%\bibitem[vZ]{vZ}  R. van Zon and Th.  Ruijgrok,
%The elliptic billiard: subtleties of separability.  European J.
%Phys. 19 (1998), no. 1, 77--84.
%
%\bibitem[WWD]{WWD} H.  Waalkens, J.  Wiersig, and H. R.  Dullin,  Elliptic quantum billiard. Ann. Physics 260 (1997), no. 1, 50--90.
%
%\bibitem[Z]{Z} S. Zelditch,  The inverse spectral problem. With an
%appendix by Johannes Sjöstrand and Maciej Zworski. Surv. Differ. Geom., IX, Surveys in differential geometry. Vol. IX, 401--467, Int. Press, Somerville, MA, 2004.

\bibitem[Z1]{Z}  S. Zelditch, Inverse spectral problem for analytic domains. II. $\Bbb Z_2$-symmetric domains. Ann. of Math. (2) 170 (2009), no. 1,
205--269.

%
%\bibitem[Z2]{Z2} S. Zelditch,  Normal forms and inverse spectral theory. Journées ``\'Equations aux D\'eriv\'ees Partielles'' (Saint-Jean-de-Monts, 1998), Exp. No. XV, 18 pp., Univ. Nantes, Nantes, 1998.
%



%\bibitem[Z3]{Z3} S. Zelditch,
%The inverse spectral problem for surfaces of revolution.  J.
%Differential Geom. 49 (1998), no. 2, 207--264.

\bibitem[Z2]{Z2} S. Zelditch,   Spectral determination of analytic bi-axisymmetric plane domains.
Geom. Funct. Anal. 10 (2000), no. 3, 628--677.

\bibitem[Z3]{Z3} S. Zelditch, Billiards and boundary traces of eigenfunctions.
Journ\'ees ``\'Equations aux Dé\'eriv\'ees Partielles'', Exp. No.
XV, 22 pp., Univ. Nantes, Nantes, 2003.



\end{thebibliography}
\end{document}